\documentclass[a4paper]{article}
\makeindex

\usepackage{mathtools, amstext, amsthm, graphicx, fullpage, xcolor, hyperref, pinlabel, subcaption, float}

\usepackage[alphabetic]{amsrefs}

\allowdisplaybreaks

\newtheorem{theorem}{Theorem}[section]
\newtheorem{lemma}[theorem]{Lemma}
\newtheorem{proposition}[theorem]{Proposition}
\newtheorem{corollary}[theorem]{Corollary}

\newtheorem{example}[theorem]{Example}
\newtheorem{definition}[theorem]{Definition}

\newcommand{\MCG}{\text{MCG}}

\definecolor{nicered}{RGB}{204,0,0}
\definecolor{niceblue}{RGB}{51,190,255}

\makeatletter
\newcommand{\address}[1]{\gdef\@address{#1}}
\newcommand{\email}[1]{\gdef\@email{\url{#1}}}
\newcommand{\@endstuff}{\par\vspace{\baselineskip}\noindent\small
\begin{tabular}{@{}l}\scshape\@address\\\textit{E-mail address:} \@email\end{tabular}}
\AtEndDocument{\@endstuff}
\makeatother

\title{Higher degree covering moves for 3-manifolds}
\author{Aru Mukherjea}
\address{Department of Mathematics, University of Texas, Austin, TX}
\email{amukherj@utexas.edu}
\date{}

\begin{document}

\maketitle
\begin{abstract}

Covering moves relate colored link diagrams appearing as the branch sets of simple branched coverings of $S^3$ by the same 3-manifold. We provide a complete set of covering moves on plat closures of braids in each fixed degree $d \geq 4$, extending prior work of Apostolakis and Piergallini. As a consequence we show that after stabilization to the same degree at least 4, only two local tangle replacements are required to relate any two colored links, recovering Bobtcheva and Piergallini's resolution of a conjecture of Montesinos. We also obtain that in the braided setting, the two local tangle replacements suffice after $d-2$ stabilizations. Lastly, we prove that the $d$-fold simple branched cover of a $d$-bridge knot is a lens space $L(p,q)$ and provide a method for determining $p$ and $q$.
\end{abstract}

\section{Introduction}

Every closed, oriented 3-manifold is a degree 3 irregular branched cover of $S^3$ with branch set a link, proven independently by Hilden \cite{hil74}, Hirsch \cite{hir74}, and Montesinos \cite{mon74}. Thus, any such 3-manifold can be fully described by a colored link in $S^3$, which is a link together with a representation of the link group into the symmetric group $S_3$. This representation, known as the \textit{monodromy} or \textit{coloring}, specifies the covering information. For a $d$-fold cover, the representation is into $S_d$.

This description is not unique: many colored links describe the same 3-manifold. In fact, for each closed, oriented 3-manifold $M$ and degree $d$, there are infinitely many colored knots that appear as the branch set of a $d$-fold branched cover $M \to S^3$ (see Figure \ref{fig:twist}).

In order to relate the different colored links describing the same manifold, we have the notion of a \textit{covering move}, or an operation performed on a colored link that preserves the covering manifold up to diffeomorphism. Two examples of covering moves are the local tangle replacements C and N shown in Figure \ref{fig:cn}. For example, the knots in Figure \ref{fig:twist1} differ by repeated applications of move C. The moves C and N were introduced in \cite{mon85a}, where it is proven they are covering moves. If the degree of the cover is allowed to vary, another pair of examples of covering moves appear: \textit{stabilization}, or increasing the degree of the cover by adding to the branch set an unlinked unknot colored $(i\ d\text{+}1)$ for some $i \leq d$, and \textit{destabilization}, or removing such a component.

In \cite{mon85b}, Montesinos conjectured that C and N are a minimal set of covering moves. This conjecture was answered in the affirmative for degree 3 by Piergallini, who constructed a set of moves for colored plats in \cite{pie91}, and showed that they could be carried out using only C, N, and link isotopy after one stabilization to degree 4 \cite{pie95}. It was also shown to be true in degree 4 (after one stabilization) by Apostolakis \cite{apo03}, and in arbitrary degree (after stabilization to degree at least 4) by Bobtcheva and Piergallini \cite{bp07, bp11}. The latter result was independent of previous techniques and proved a correspondence between 4-dimensional 2-handlebodies (up to 2-equivalence) and colored ribbon surfaces in $B^4$ appearing as their branch sets (up to a set of covering moves), providing a link between Kirby calculus in the cover and combinatorial moves on surfaces in the target. By restricting to the boundary, they obtain the result for links in $S^3$.

We extend the original technique used by Piergallini and Apostolakis to arbitrary degree, with key ingredient the generators for the liftable subgroup provided by Wajnryb and Wi\'{s}niowska-Wajnryb for degree 4 in \cite{ww08} and higher degrees in \cite{ww12}.

\begin{theorem} \label{thm:moves}
Suppose two colored plats $L_1, L_2 \subset S^3$ are such that the degree $d \geq 3$ simple branched covering of $S^3$ over each $L_i$ is the same closed 3-manifold. Then, $L_1$ and $L_2$ can be related by colored isotopy and a finite sequence of moves N, C, II$_i$, III$_{i,j}$, and IV, which preserve the plat form of the link.
\end{theorem}

\bigskip
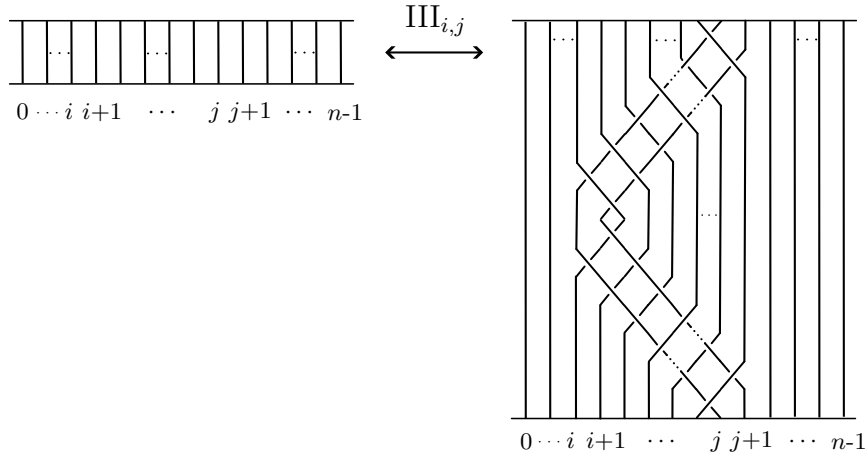
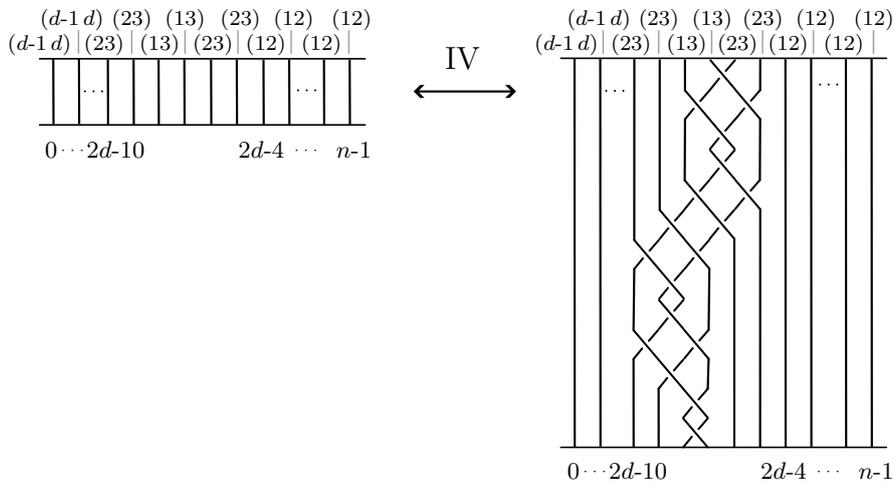
\begin{figure}
\centering

\begin{subfigure}[h]{0.5\textwidth}
\labellist
\hair 2pt
\large\pinlabel C at 62 60
\pinlabel N at 290 60
\small\pinlabel $(i\ j)$ at -5 100
\pinlabel $(j\ k)$ at 32 100
\pinlabel $(i\ j)$ at -5 -10
\pinlabel $(j\ k)$ at 32 -10
\pinlabel $(i\ j)$ at 95 100
\pinlabel $(j\ k)$ at 132 100
\pinlabel $(i\ j)$ at 95 -10
\pinlabel $(j\ k)$ at 132 -10
\pinlabel $(i\ k)$ at 150 30

\pinlabel $(i\ j)$ at 220 100
\pinlabel $(k\ l)$ at 257 100
\pinlabel $(i\ j)$ at 220 -10
\pinlabel $(k\ l)$ at 257 -10
\pinlabel $(i\ j)$ at 320 100
\pinlabel $(k\ l)$ at 357 100
\pinlabel $(i\ j)$ at 320 -10
\pinlabel $(k\ l)$ at 357 -10
\endlabellist
\includegraphics[width=\textwidth]{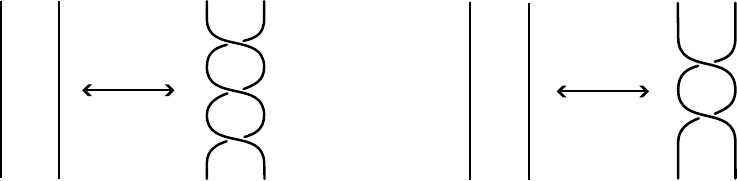}
\medskip
\caption{Moves C and N.}
\label{fig:cn}
\end{subfigure}

\bigskip

\begin{subfigure}[h]{0.7\textwidth}
\labellist
\large
\pinlabel II$_i$ at 185 28
\pinlabel II$_0$ at 185 115
\small
\pinlabel $0$ at 7 -10
\pinlabel $\cdots$ at 21 -10
\pinlabel $2i$ at 33 -10
\pinlabel $2i\text{+}1$ at 52 -10.5
\pinlabel $\cdots$ at 85 -10
\pinlabel $n\text{-}2$ at 118 -10
\pinlabel $n\text{-}1$ at 137 -10

\pinlabel $0$ at 233 -10
\pinlabel $\cdots$ at 247 -10
\pinlabel $2i$ at 261 -10
\pinlabel $2i\text{+}1$ at 280 -10.5
\pinlabel $\cdots$ at 311 -10
\pinlabel $n\text{-}2$ at 344 -10
\pinlabel $n\text{-}1$ at 363 -10
\pinlabel $0$ at 7 75
\pinlabel $1$ at 21 75
\pinlabel $\cdots$ at 69 75
\pinlabel $n\text{-}2$ at 118 75
\pinlabel $n\text{-}1$ at 137 75

\pinlabel $0$ at 233 75
\pinlabel $1$ at 247 75
\pinlabel $\cdots$ at 295 75
\pinlabel $n\text{-}2$ at 344 75
\pinlabel $n\text{-}1$ at 363 75
\endlabellist

\includegraphics[width=\textwidth]{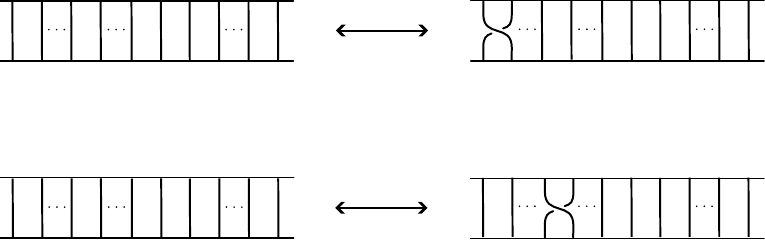}
\medskip
\caption{Moves II$_i$. $i$ runs from $0$ to $d-3$. All strands are standardly colored.}
\end{subfigure}

\bigskip\bigskip

\begin{subfigure}[h]{0.7\textwidth}
\labellist
\large\pinlabel III$_{i,j}$ at 193 182
\tiny\pinlabel $\cdots$ at 18 140
\pinlabel $\cdots$ at 245 -10
\small\pinlabel $0$ at 6 140.5
\pinlabel $i$ at 27 140.5
\pinlabel $i\text{+}1$ at 42 140
\pinlabel $\cdots$ at 70 140
\pinlabel $j$ at 93 140
\pinlabel $j\text{+}1$ at 109 140
\pinlabel $\cdots$ at 132 140
\pinlabel $n\text{-}1$ at 152 140

\pinlabel $0$ at 234 -10
\pinlabel $i$ at 254 -9.5
\pinlabel $i\text{+}1$ at 270 -10
\pinlabel $\cdots$ at 296 -10
\pinlabel $j$ at 320 -10
\pinlabel $j\text{+}1$ at 336 -10
\pinlabel $\cdots$ at 360 -10
\pinlabel $n\text{-}1$ at 380 -10
\endlabellist
\includegraphics[width=\textwidth]{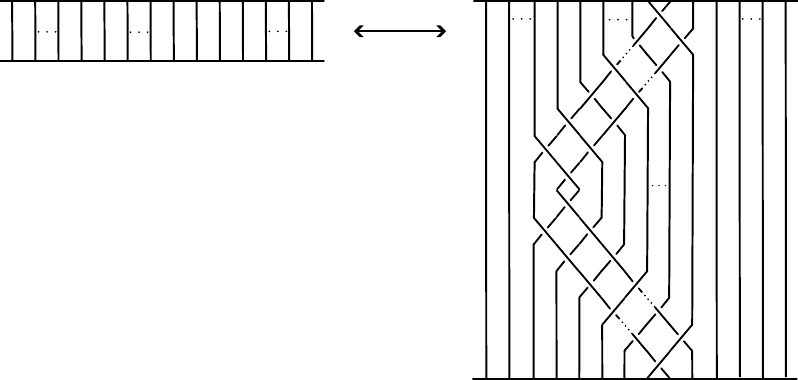}
\medskip
\caption{Moves III$_{i,j}$. $i,j$ are even and $j \leq 2d-4$. All strands are standardly colored.}
\bigskip\bigskip
\end{subfigure}

\begin{subfigure}[h]{0.7\textwidth}
\labellist
\large\pinlabel IV at 183 170
\footnotesize
\pinlabel $(d\text{-}1\thinspace d)$ at 0 176
\pinlabel $(d\text{-}1\thinspace d)$ at 14 187
\pinlabel $(23)$ at 28 176
\pinlabel $(23)$ at 40 187
\pinlabel $(13)$ at 52 176
\pinlabel $(13)$ at 63 187
\pinlabel $(23)$ at 75 176
\pinlabel $(23)$ at 87 187
\pinlabel $(12)$ at 98 176
\pinlabel $(12)$ at 110 187
\pinlabel $(12)$ at 122 176
\pinlabel $(12)$ at 135 187

\pinlabel $(d\text{-}1\thinspace d)$ at 228 176
\pinlabel $(d\text{-}1\thinspace d)$ at 242 187
\pinlabel $(23)$ at 256 176
\pinlabel $(23)$ at 268 187
\pinlabel $(13)$ at 280 176
\pinlabel $(13)$ at 291 187
\pinlabel $(23)$ at 302 176
\pinlabel $(23)$ at 313 187
\pinlabel $(12)$ at 324 176
\pinlabel $(12)$ at 336 187
\pinlabel $(12)$ at 348 176
\pinlabel $(12)$ at 361 187
\small
\pinlabel \textcolor{gray}{$\mid$} at 17 177
\pinlabel \textcolor{gray}{$\mid$} at 40 177
\pinlabel \textcolor{gray}{$\mid$} at 63 177
\pinlabel \textcolor{gray}{$\mid$} at 86 177
\pinlabel \textcolor{gray}{$\mid$} at 109 177
\pinlabel \textcolor{gray}{$\mid$} at 134 177

\pinlabel \textcolor{gray}{$\mid$} at 244 177
\pinlabel \textcolor{gray}{$\mid$} at 268 177
\pinlabel \textcolor{gray}{$\mid$} at 291 177
\pinlabel \textcolor{gray}{$\mid$} at 313 177
\pinlabel \textcolor{gray}{$\mid$} at 335 177
\pinlabel \textcolor{gray}{$\mid$} at 361 177

\pinlabel $0$ at 5 130
\pinlabel $2d\text{-}10$ at 33 130.5
\pinlabel $2d\text{-}4$ at 96 130.5
\pinlabel $n\text{-}1$ at 136 130
\pinlabel $0$ at 231 -10
\pinlabel $2d\text{-}10$ at 259 -9.5
\pinlabel $2d\text{-}4$ at 322 -9.5
\pinlabel $n\text{-}1$ at 362 -10
\tiny\pinlabel $\cdots$ at 15 130
\pinlabel $\cdots$ at 116 130
\pinlabel $\cdots$ at 241 -10
\pinlabel $\cdots$ at 342 -10

\endlabellist
\includegraphics[width=\textwidth]{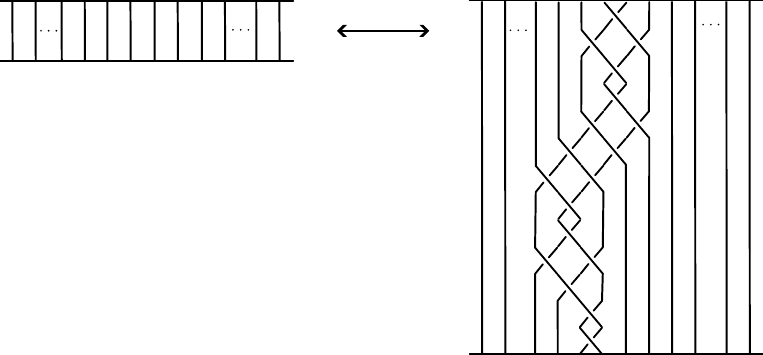}
\medskip
\caption{Move IV.}
\label{fig:iv}
\end{subfigure}

\caption{The moves in Theorem \ref{thm:moves}. The labelling conventions used in these diagrams are described in Section \ref{bg:conventions}. The standard coloring is described in Figure \ref{fig:std}.}
\label{fig:moves}
\end{figure}

\begin{figure}
\centering 

\begin{subfigure}[h]{0.36\textwidth}
\labellist\small
\pinlabel $6n$ at 34 56
\pinlabel $\beta$ at 75 15
\pinlabel $(2\ 3)$ at 14 80
\pinlabel $(1\ 2)$ at 52 80
\pinlabel $\cdots$ at 90 80
\pinlabel $(1\ 2)$ at 142 80
\endlabellist
\includegraphics[width=\textwidth]{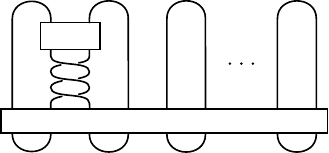}
\caption{}
\label{fig:twist1}
\end{subfigure} \qquad\qquad
\begin{subfigure}[h]{0.53\textwidth}
\labellist\small
\pinlabel $6n_{d-2}$ at 29 55.5
\pinlabel $6n_2$ at 79 56
\pinlabel $6n_1$ at 117 56
\pinlabel $\beta$ at 150 15
\pinlabel $(d\text{-}1\ d)$ at 7 80
\pinlabel $\cdots$ at 55 80
\pinlabel $(2\ 3)$ at 97 80
\pinlabel $(1\ 2)$ at 134 80
\pinlabel $\cdots$ at 166 80
\pinlabel $(1\ 2)$ at 212 80
\endlabellist
\includegraphics[width=\textwidth]{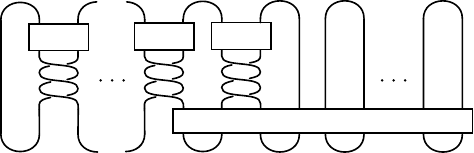}
\caption{}
\end{subfigure}

\caption{Examples of colored knots with the same cover.\smallskip \\ 
(a) Given a 3-manifold $M$, take a colored knot that describes $M$ as a 3-fold simple branched cover of $S^3$. Following \cite{hil74}, this can be written as shown in the figure, for some braid $\beta$. By varying the number of half-twists $n$ in the depicted region, we obtain a family of knots, infinitely many of which are distinct by \cite{kms92}.
\smallskip \\ 
(b) This family of knots can be obtained by stabilizing the construction in Figure \ref{fig:twist1} and joining components using move C. Twist regions about adjacent strands can be added using move C. Twist regions about non-adjacent strands (not pictured) can be added using move N. We obtain infinite families of distinct knots with the same cover, in each degree. For degree at least 4, the twist regions can be chosen so that the family contains knots of arbitrarily high genus \cite{bm19}. As degree of a connected covering is a lower bound for bridge number, successive stabilizations (and joining components using move C) provide a family of knots of increasing bridge number that have the same cover.}
\label{fig:twist}
\end{figure}

For $d=3$, the moves of Theorem \ref{thm:moves} coincide with \cite{pie91}. Other than C and N, the moves require a particular plat form presentation of the entire link.

As a consequence, we obtain a completely 3-dimensional proof of Montesinos' conjecture, recovering the earlier proof in \cite{bp07}.

\begin{theorem}\label{thm:cn}
Suppose two colored links $L_1, L_2 \subset S^3$ are such that the simple branched covering of $S^3$ over each $L_i$ is the same closed 3-manifold. Then, after stabilization to the same degree at least 4, $L_1$ and $L_2$ can be related by colored isotopy and a finite sequence of moves N and C.
\end{theorem}

The approach in Theorem \ref{thm:moves} of using plat closures is natural to the setting of branched covers, as the extra structure induces a Heegaard splitting of the covering 3-manifold. The moves preserve bridge position of the link, so are also applicable to covers of trivial tangles in $B^3$. In particular, up to move C, moves II-IV are positive braids.
As each move is up to braid rather than link isotopy, we may apply them fiberwise to diagrammatically address lifting questions in settings other than links in $S^3$. One example is the setting of Lefschetz fibered 4-manifolds (see \cite{lp01}, \cite{zud07}, \cite{hug22}, \cite{apz11}, \cite{pz18}). The results of \cite{ww12} stated in Proposition \ref{prop:l-gens} address branched coverings of $D^2$ with total monodromy other than the identity, which could be used to study the analogous problem of covering moves for manifolds with boundary. 

The reverse question may also be asked: given a class of links, what manifolds appear as their irregular covers? An upper bound for the Heegaard genus of the covering manifold can be computed from the degree of the cover and the bridge number of the link. For low Heegaard genus, as a corollary of the above theorem, we obtain diagrammatic proofs of the following, extending the result of Uchida \cite{uch05} in degree 3.

\begin{corollary}\label{cor:lens} Let $M$ be the $d$-fold simple branched cover over the colored $b$-bridge link $L$ in $S^3$. Then
\begin{itemize}
\item[(a)] If $b=d$, $M$ is a lens space $L(p,q)$, and $p$ and $q$ can be determined diagrammatically.
\item[(b)] If $b=d+1$, $M$ can be obtained as the double branched cover over a 3-bridge link.
\end{itemize}
\end{corollary}

Note that for the cover to be connected, $b$ must be at least $d-1$. The latter result also follows from the fact that the every element of the mapping class group of a genus 2 surface commutes with the hyperelliptic involution. It would be interesting to have a more specific answer to this question for Heegaard genus 2 or greater.

Results about 3-dimensional covering moves also have 4-dimensional interest. The moves C and N are particularly desirable as they can be carried out as cobordisms between colored tangles $(B^3, T) \subset (S^3, L)$ that lift to product cobordisms $B^3 \times I$. Further motivation comes from \cite{mon85b}, where it is shown that every 4-dimensional 2-handlebody, or 4-manifold built using only 0, 1, and 2-handles, is a 3-fold branched cover of $B^4$ over a ribbon surface.
Montesinos' conjecture then implies that branched covers by 4-manifolds could be glued together along their boundaries, with a possibly singular branch set. \cite{pie95} thus shows that every closed oriented 4-manifold is a degree 4 branched cover of $S^4$ with immersed branch set. Iori and Piergallini \cite{ip02} improve the branch set to be embedded, at the cost of stabilizing to degree 5. It remains open whether the branch set can always be made embedded for degree 4.

It is then useful to keep track of which cobordisms can be achieved by a covering move.
The role of stabilization in \cite{pie95} is to introduce an extra unlinked unknotted component that can be used to carry out moves that break (and restore) the plat form of the link, offering more flexibility. In particular, an appropriately colored stabilizing unknot can realize certain pairs of band sums on the link, providing information about the resulting cobordism. In the following theorem, we relate stabilizations of the moves in Theorem \ref{thm:moves}.

\begin{theorem} \label{thm:sheets}
After the addition of $d-2$ trivial sheets colored $(d\text{-}j\ d\text{+1+}j)$ for $0 \leq j \leq d - 3$, the braids of moves II$_i$, III$_{i,j}$, and IV in Theorem \ref{thm:moves} can be obtained from the trivial braid using only moves N, C, and isotopy. In other words, after $d-2$ stabilizations, moves II$_i$, III$_{i,j}$, and IV can be removed from the statement of Theorem \ref{thm:moves}.
\end{theorem}

The covering moves for ribbon surfaces in $B^4$ of \cite{bp07} are refined in \cite{apz11} and \cite{pz18} to moves preserving braided position of the ribbon surface, and therefore braid closure position of the colored link comprising its boundary. It would be interesting to adapt these moves to instead preserve plat position of the link in the boundary and thereby a Heegaard splitting of the boundary of the cover.

A main motivating question behind the study of covering moves for 3-manifolds remains open, namely if every closed, oriented 4-manifold is a degree 4 branched cover of $S^4$ with embedded branch set. A set of covering moves for colored surfaces in $S^4$ covered by diffeomorphic closed 4-manifolds has also not yet been found. Such a result may provide progress towards exhibiting 4-manifolds as 4-fold covers of $S^4$ with embedded branch set by relating arbitrary degree 5 branched covers to those arising as the stabilization of a degree 4 cover branched over an embedded surface. Covering moves for surfaces in $S^4$ could also be used to construct branched covers of closed 5-manifolds and thereby show that there exists some $n$ for which every (closed, PL, orientable) 5-manifold is a $n$-fold branched cover of $S^5$, analogously to the 4-dimensional case.

\subsubsection*{Organization} In Section \ref{bg}, we establish conventions and background on braids and branched covers. Section \ref{pf:outline} is an outline of the proof strategy. Sections \ref{pf:thm1}, \ref{pf:cn}, \ref{pf:sheets}, and \ref{pf:lens} contain the proofs of Theorems \ref{thm:moves}, \ref{thm:cn}, \ref{thm:sheets}, and Corollary \ref{cor:lens} respectively.

\subsubsection*{Acknowledgements} 

This work was partially supported by NSF grant DMS-2404810. The author would like to thank Maggie Miller and Sashka Kjuchukova for the helpful advice and discussion, as well as Riccardo Piergallini for explanation regarding \cite{bp07}.

\section{Preliminaries}\label{bg}

\subsection{Braid words and half-twists}

We denote by $B_n$ the braid group on $n$ strands.
$B_n$ can be viewed as all possible \textit{geometric braids}, or submanifolds of $D^2 \times I$, up to isotopy, such that the projection onto the $I$ factor is a covering map. Each braid is a cobordism between $(D^2, n\text{ points})$ and itself.
Alternatively, $B_n$ is the mapping class group of a disk with $n$ marked points $\{A_0, \ldots, A_{n-1}\}$. One generating set for $B_n$ is the \textit{Artin generators} $\{\sigma_0, \ldots, \sigma_{n-2}\}$. From the geometric braid perspective, $\sigma_i$ denotes a positive crossing between strands $i$ and $i\text{+}1$, as in Figure \ref{fig:artin gen}.

\medskip
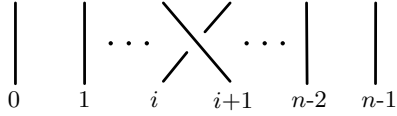
\begin{figure}
\labellist
\small\pinlabel $0$ at 0 -5
\pinlabel $1$ at 24 -5
\pinlabel $i$ at 48 -5
\pinlabel $i\text{+}1$ at 75 -5.5
\pinlabel $n\text{-}2$ at 101 -5
\pinlabel $n\text{-}1$ at 125 -5
\endlabellist

\centering
\includegraphics[width=0.3\textwidth]{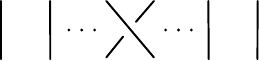}
\medskip
\caption{The $i$-th Artin generator $\sigma_i$.}
\label{fig:artin gen}
\end{figure}

From the mapping class perspective, $\sigma_i$ is defined with respect to a \textit{Hurwitz system}. 

\begin{definition}
Choose a basepoint $A \in \partial D^2$. A \textit{Hurwitz system} is a set of $n$ arcs $\alpha_0, \ldots, \alpha_{n-1}$ intersecting only at $A$, where $\alpha_i$ has endpoints $A$ and $A_i$.
\end{definition}

A Hurwitz system corresponds to a generating set for $\pi_1(D^2 - \{A_0, \ldots, A_{n-1}\})$ where each generator follows an arc $\alpha_i$, goes clockwise around the endpoint $A_i$, then returns along the arc to $A$. We will abuse notation and use $\alpha_i$ to refer to both the Hurwitz arcs and their corresponding generators of $\pi_1(D^2 - \{A_0, \ldots, A_{n-1}\})$. From a Hurwitz system $\alpha_0, \ldots, \alpha_{n-1}$, we can construct a \textit{maximal chain} of arcs $x_0, \ldots, x_{n-1}$, where $x_i$ has endpoints $A_i$ and $A_{i+1}$, and the $\{x_i\}$ do not intersect outside of their shared endpoints. Now, $\sigma_i$ can be defined as a positive half-twist about the arc $x_i$. This also defines the action of $B_n$ on $D^2 - \{A_0, \ldots, A_{n-1}\}$.

In what follows, we will further abuse notation and refer interchangeably to arcs connecting branch points in $D^2$ and the positive half-twist about them, for instance by referring to the generator $\sigma_i$ as $x_i$. 

\begin{figure}
\centering

\begin{subfigure}[b]{0.43\textwidth}
\labellist
\small\pinlabel $\gamma$ at 50 60
\pinlabel $\alpha_0$ at 12 40
\pinlabel $\alpha_1$ at 87 40
\pinlabel $\gamma(\alpha_0)$ at 180 32
\pinlabel $\gamma(\alpha_1)$ at 225 67
\endlabellist
\includegraphics[width=\textwidth]{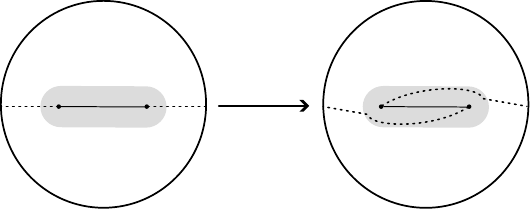}
\caption{}
\end{subfigure}
\hfill
\begin{subfigure}[b]{0.53\textwidth}
\labellist
\small\pinlabel $\gamma$ at 50 70
\pinlabel $\alpha_0$ at 25 30
\pinlabel $\alpha_1$ at 85 20

\pinlabel $\gamma(\alpha_0)$ at 230 75
\pinlabel $\gamma(\alpha_1)$ at 262 20
\endlabellist
\includegraphics[width=\textwidth]{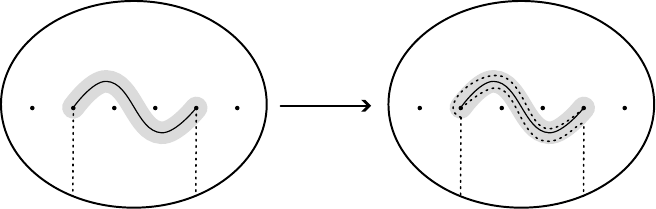}
\caption{}
\end{subfigure}

\caption{Examples of half-twists about different arcs $\gamma$. The positive half-twist is defined on a disk neighbourhood of an arc by rotating it 180 degrees anticlockwise, fixing the arc setwise and exchanging the endpoints.}
\end{figure}
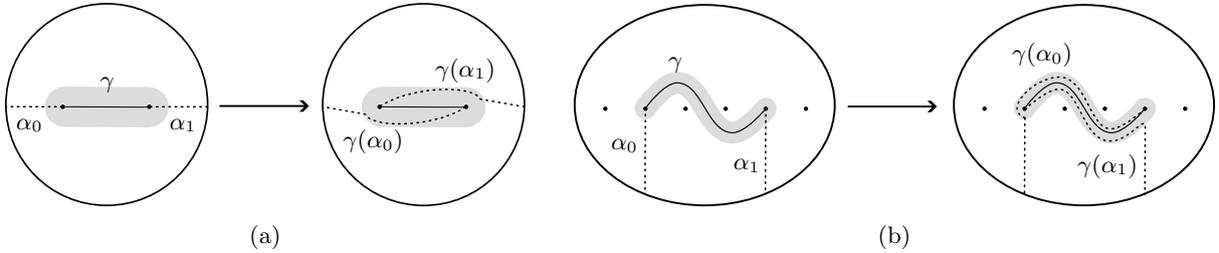

\begin{figure}
\centering

\begin{subfigure}[b]{0.35\textwidth}
\labellist
\small\hair 2pt
\pinlabel $x_0$ at 25 55
\pinlabel $x_1$ at 45 55
\pinlabel $x_2$ at 65 55
\pinlabel $x_3$ at 85 55
\pinlabel $x_4$ at 105 55

\pinlabel $\alpha_0$ at 13 35
\pinlabel $\alpha_1$ at 30 38
\pinlabel $\alpha_2$ at 45 40
\pinlabel $\alpha_3$ at 60 40
\pinlabel $\alpha_4$ at 74 38
\pinlabel $\alpha_5$ at 102 35

\endlabellist
\includegraphics[width=\textwidth]{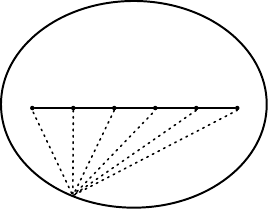}
\caption{}
\end{subfigure}
\quad\quad
\begin{subfigure}[b]{0.35\textwidth}
\labellist
\small\hair 2pt
\pinlabel $x_0$ at 62 76
\pinlabel $x_1$ at 55 57
\pinlabel $x_2$ at 67 45
\pinlabel $x_3$ at 76 40
\pinlabel $x_4$ at 85 58

\pinlabel $\alpha_5$ at 19 37
\pinlabel $\alpha_1$ at 31 30
\pinlabel $\alpha_3$ at 44 30
\pinlabel $\alpha_2$ at 60 30
\pinlabel $\alpha_4$ at 77 28.5
\pinlabel $\alpha_0$ at 97 27

\endlabellist
\includegraphics[width=\textwidth]{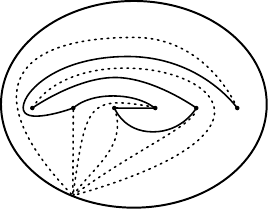}
\caption{}
\end{subfigure}

\caption{A simple and more complicated Hurwitz system $\alpha_0,\ldots ,\alpha_{n-1}$ and corresponding maximal chain $x_0,\ldots ,x_{n-2}$ for a disk with six punctures.}
\end{figure}
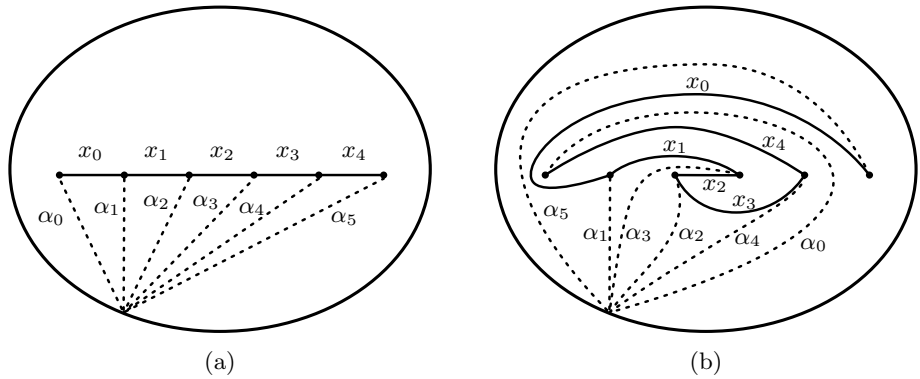

Given a half-twist about an arc from $A_i$ to $A_j$, we obtain the corresponding geometric braid by tracing out the motion of each marked point under the half-twist. We can then read off a decomposition into Artin generators, which will be a conjugate of one of the generators by a word describing the arc (see Figure \ref{fig:braidword}).

\begin{figure}[t]
\labellist
\large\hair 2pt
\pinlabel $D^2$ at 148 102
\pinlabel $\times I$ at 145 60
\small\pinlabel $A_0$ at 15 105
\pinlabel $A_3$ at 75 105
\endlabellist

\centering
\includegraphics[width=0.3\textwidth]{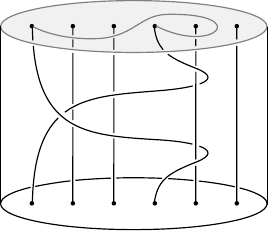}
\caption{The half-twist about this arc from $A_0$ to $A_3$ corresponds to the braid word $[x_0]x_1x_2{x_3}^{-2}$.}
\label{fig:braidword}
\end{figure}
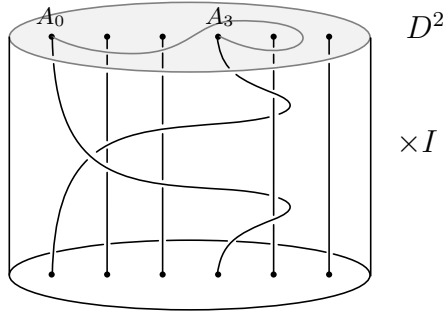

We will use these two perspectives interchangeably throughout the paper.

\subsection{Monodromy and coloring}\label{bg:color}

A branched covering $p:M \to N$ is determined by its degree $d$, branch set $N_0 \subset N$, and monodromy representation $\rho^p: \pi_1(N - N_0) \to S_d$.

When $N=D^2$, a sequence of $n$ elements of $S_d$ corresponding to the image of each generator in a Hurwitz system under $\rho^p$ is called a \textit{monodromy sequence} for that Hurwitz system. The \textit{monodromy} or \textit{coloring} of a point will refer to the image of its corresponding generator. The product of the monodromies of each generator, or equivalently the monodromy of a loop isotopic into the boundary, is called the \textit{total monodromy}. 

\begin{theorem}[\cite{be79}, \cite{mp01}] \label{thm:unique}
Branched covers of $D^2$ are determined by their degree, number of branch points, and conjugacy class of the total monodromy in $S_d$.
\end{theorem}

Branched covers of $S^2$ are exactly branched covers of $D^2$ with total monodromy the identity in $S_d$. Since we will only work with covers of $S^2$, the above theorem implies that for a fixed degree and number of branch points, there is a unique branched cover of $S^2$.

Similarly, when $N=S^3$, $N_0$ is a link, and the monodromy can be specified by labelling each arc in a diagram of $N_0$ by the image under $\rho^p$ of the Wirtinger generator corresponding to that arc. This labelling must respect the Wirtinger relations, meaning that at every crossing, the label of the under-strand is conjugated by the label of the over-strand. Such a labelling is called a \textit{coloring} and the data $(N_0, \rho^p)$ a \textit{colored link}. Here as well the terms monodromy and coloring will be used interchangeably. A coloring of a link diagram uniquely determines a coloring of any diagram obtained from it by Reidemeister moves.

To compute branched covers it is convenient for links to be written as plat closures of braids, where instead of the usual closure, each adjacent pair of strands is joined by an arc at the top and bottom, as depicted in Figure \ref{fig:plat}. Every link has a plat presentation (see, for example, \cite{bz85}). If a link admits a particular coloring, it is enough to specify that coloring on the top (or bottom) plat arcs.

\begin{figure}
\centering
\includegraphics[width=0.7\textwidth]{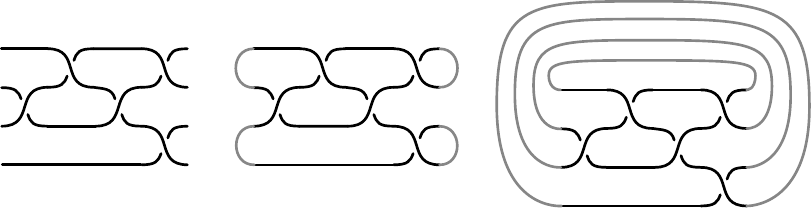}
\caption{A braid, its plat closure, and its standard closure.} \label{fig:plat}
\end{figure}

We focus on simple branched covers, where every point in the branch set has only one point in its preimage where branching occurs, and that point has branching index 2. Equivalently, the monodromy of each generator is a transposition in $S_d$.

The \textit{standard} degree $d$ coloring that we will use is defined in Figure \ref{fig:std}. By Theorem \ref{thm:unique} any colored link describing a connected degree $d$ simple branched cover of $S^2$ can be put in this form. We will see in Section \ref{pf:standard} that only isotopy is required to do so. Up to conjugation inside $S_d$, this coincides with the definition of the standard coloring in \cite{pie91} and Definition 15 of \cite{ww12} (see also \cite{ww12} Corollary 20), but differs from that of \cite{apo03}.

\begin{figure}
\labellist
\small\hair 2pt
\pinlabel $0$ at 24 -5
\pinlabel $1$ at 51 -5
\pinlabel $2$ at 79 -5
\pinlabel $3$ at 106 -5
\pinlabel $2d-6$ at 158 -5
\pinlabel $2d-5$ at 186 -5
\pinlabel $2d-4$ at 214 -5
\pinlabel $2d-3$ at 242 -5
\pinlabel $n-2$ at 298 -5
\pinlabel $n-1$ at 326 -5

\pinlabel $(d\text{-}1\ d)$ at 22 23
\pinlabel $(d\text{-}1\ d)$ at 49 23
\pinlabel $(d\text{-}2\ d\text{-}1)$ at 79 23
\pinlabel $(d\text{-}2\ d\text{-}1)$ at 110 23
\pinlabel $(2\ 3)$ at 161 23
\pinlabel $(2\ 3)$ at 189 23
\pinlabel $(1\ 2)$ at 216 23
\pinlabel $(1\ 2)$ at 244 23
\pinlabel $(1\ 2)$ at 298 23
\pinlabel $(1\ 2)$ at 326 23

\normalsize\pinlabel \textcolor{gray}{$B^3$} at -10 20
\pinlabel $S^2$ at -10 2
\endlabellist

\centering
\bigskip
\includegraphics[width=0.9\textwidth]{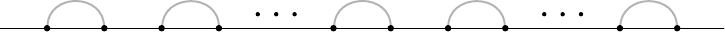}
\medskip
\caption{The standard degree $d$ coloring, shown bounding plat arcs in $B^3$.
For the first $2d-4$ branch points, color the first two by $(d\text{-}1\ d)$, the next two by $(d\text{-}2\ d\text{-}1)$, the next two by $(d\text{-}3\ d\text{-}2)$, and so on until $(2\ 3)$. The remaining branch points are colored $(1\ 2)$.} \label{fig:std}
\end{figure}
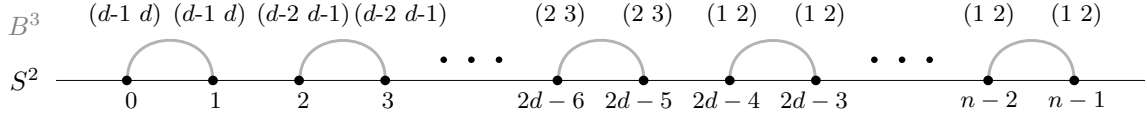

\subsection{Liftability and the liftable subgroup of $B_n$} \label{bg:liftable}

A colored link in plat form can be considered to be in bridge position inside $S^3$, inducing a Heegaard splitting of the covering manifold of some genus $g$. The gluing map, an element of the mapping class group of the genus $g$ surface $\MCG(\Sigma_g)$, is determined by the braid comprising the plat.

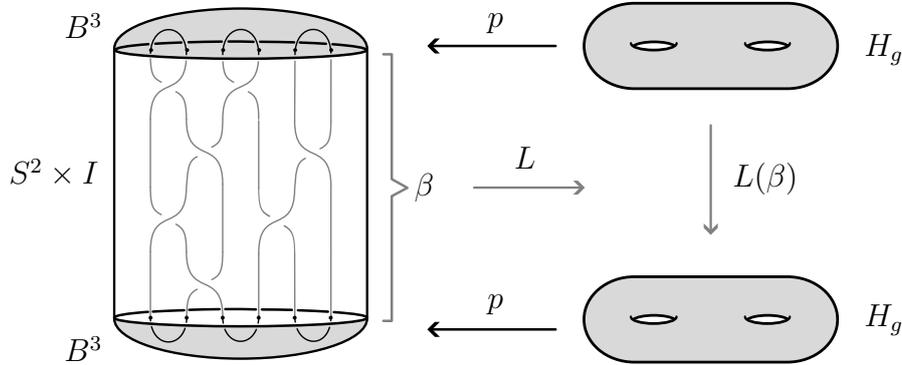
\begin{figure}
\labellist
\large\hair 2pt
\pinlabel $B^3$ at -10 115
\pinlabel $S^2\times I$ at -20 65
\pinlabel $B^3$ at -10 5
\pinlabel $L$ at 139 70
\pinlabel $p$ at 129 116
\pinlabel $p$ at 129 20

\pinlabel $H_g$ at 260 107
\pinlabel $H_g$ at 260 15
\pinlabel $\beta$ at 105 60
\pinlabel $L(\beta)$ at 220 63
\endlabellist

\centering
\includegraphics[width=0.6\textwidth]{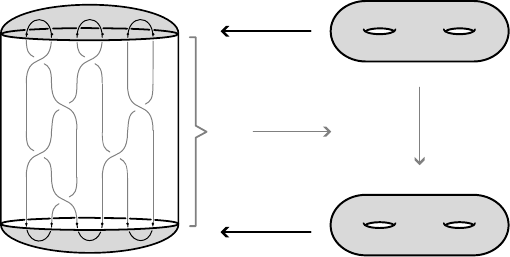}
\caption{Each copy of $B^3$ branched over a set of standardly colored plat arcs is covered by a genus $g=\frac{n}{2}-(d-1)$ handlebody. The braid $\beta$ can be viewed as the path traced out by the marked points under the mapping class $\beta$, which lifts under the branched covering map to some $L(\beta) \in \MCG(\Sigma_g)$.}
\end{figure}

In order for $L(\beta)$ to be well-defined as an element of $\MCG(\Sigma_g)$, the colored braid must be \emph{liftable}.
This will always be true if $\beta$ comes from a colored plat appearing as the branch set of some branched covering, but may not be true for any braid in $B_n$.
Here we will view $B_n$ as the mapping class group of $D^2$ with $n$ marked points.

\begin{definition} A braid $\beta$ is liftable with respect to a branched covering $p$ if there is some $\varphi \in \MCG(\Sigma_g)$ such that $\beta \circ p = p \circ \varphi$. 
Equivalently, $\beta$ is liftable if it preserves the monodromy sequence of some Hurwitz system.
\end{definition}

In terms of a braid diagram, a braid is liftable if the monodromy sequence at the top and the bottom of the braid are the same. 
This guarantees that we can extend the coloring to the plat arcs in each $B^3$.
Intuitively, this condition ensures the same sheets will remain adjacent in the cover.
For a fixed covering map $p$, liftable braids form a finite-index subgroup $L^p(B_n)$ of $B_n$ and the map $L:L^p(B_n) \to \MCG(\Sigma_g)$ is a group homomorphism. As some power of every half-twist is liftable, we can classify half-twists by their minimal positive liftable power.

\begin{definition} A half-twist $t$ about an arc with endpoint $A_i$ in $D^2 - \{A_0, \ldots, A_{n-1}\}$ is called \textit{Type j} if $\rho^p(t(\alpha_i))\rho^p(\alpha_i)$ has order $j$. 
\end{definition}

Being Type 1 is equivalent to being liftable. For simple branched coverings, we must have $j \in \{1,2,3\}$. Type 1 occurs if $\rho^p(t(\alpha_i)) = \rho^p(\alpha_i)$, Type 2 if the transpositions $\rho^p(t(\alpha_i))$ and $\rho^p(\alpha_i)$ share no symbols, and Type 3 if $\rho^p(t(\alpha_i))$ and $\rho^p(\alpha_i)$ have one common symbol. Thus $\rho^p(t^j(\alpha_i))$ will equal $\rho^p(\alpha_i)$, and $t^j$ is liftable.

Birman and Wajnryb \cite{bw85} provided a small generating set for the liftable subgroup of degree 3 simple covers of $D^2$ branched over $n$ points, consisting of the minimal liftable powers of half-twists about $n$ arcs. Mulazzani and Piergallini \cite{mp01} showed that for any branched covering of $D^2$, the liftable subgroup is finitely generated by liftable powers of arcs. Apostolakis \cite{apo03} computed generators for degree 4 simple covers of $D^2$. Wajnryb and Wi\'{s}niowska-Wajnryb \cite{ww12} give a small generating set for the liftable subgroup of simple branched covers of $D^2$ of any degree. We state the latter result, which is a key part of our proof.

\begin{proposition}[\cite{ww12} Theorem 31]\label{prop:l-gens}
The liftable subgroup $L^p(B_n)$ is generated by the minimal positive liftable powers of half-twists about the arcs $x_i, y_{i,j}, z_{i,j}$ of type 2, $w_{i,j}$, and special arcs $s_i$ depending on $p$, given in Figure \ref{fig:arcs}.
\end{proposition}

\begin{figure}
\centering
\begin{subfigure}[b]{0.49\textwidth}
	\labellist
	\small
	\pinlabel $x_i$ at 43 19
	\pinlabel $y_{i,j}$ at 85 5
	\pinlabel $z_{i,j}$ at 85 31
	\pinlabel $A_i$ at 26 7
	\pinlabel $A_{i\text{+}1}$ at 57 7
	\pinlabel $A_j$ at 113 7
	\pinlabel $A_{j\text{+}1}$ at 141 7
	\endlabellist

	\includegraphics[width=\textwidth]{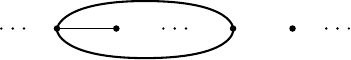}
	\caption{Arcs $x_i$, $y_{i,j}$, and $z_{i,j}$.}
\end{subfigure}

\bigskip\bigskip

\begin{subfigure}[b]{0.7\textwidth}
	\labellist
	\small
	\pinlabel $A_0$ at 10 35
	\pinlabel $A_1$ at 38 35
	\pinlabel $A_2$ at 66 35
	\pinlabel $A_3$ at 94 35
	\pinlabel $A_4$ at 123 35
	\pinlabel $A_5$ at 151 35
	\pinlabel $A_6$ at 180 35
	\pinlabel $A_7$ at 208 35
	\pinlabel $w_{0,2}$ at 30 67
	\pinlabel $w_{2,6}$ at 85 67
	
	\pinlabel $A_0$ at 10 -8
	\pinlabel $A_1$ at 38 -8
	\pinlabel $A_2$ at 66 -8
	\pinlabel $A_3$ at 94 -8
	\pinlabel $A_4$ at 123 -8
	\pinlabel $A_5$ at 151 -8
	\pinlabel $A_6$ at 180 -8
	\pinlabel $A_7$ at 208 -8
	\pinlabel $w_{0,6}$ at 25 24
	\endlabellist

	\includegraphics[width=\textwidth]{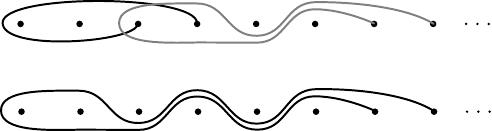}
	\medskip
	\caption{Examples of arcs $w_{i,i+2}$ and $w_{i,i+4}$ (above) and $w_{i,i+6}$ (below) for a standard cover with $d \geq 5$.}
	
\end{subfigure}

\bigskip

\begin{subfigure}[b]{0.56\textwidth}
	\labellist
	\small
	\pinlabel $A_{2d\text{-}4}$ at 85 7
	\pinlabel $A_{2d\text{-}1}$ at 169 7
	\endlabellist
	
	\includegraphics[width=\textwidth]{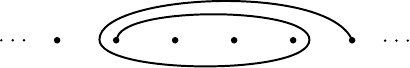}
	\caption{Arc $s_1$.}
\end{subfigure}
\caption{Generating arcs from Proposition \ref{prop:l-gens}.\smallskip \\ 
(a) $x_i$ are the arcs of the maximal chain of a Hurwitz system. $y_{i,j}$ (respectively $z_{i,j}$) starts at $A_i$ and ends at $A_j$, travelling only rightwards and passing in front of (respectively behind) all points $A_k$ in between.
\smallskip \\ (b) $w_{i,j}$ is defined when $i < j-1$ and $x_i,x_j$ are both of Type 1. The arc starts at $A_j$, travels to the left, passing below $A_k$ for which $x_k$ is Type 1 and above other $x_k$. It passes below $x_{i+1}$ and $x_i$, then up and to the right above $x_i$ and $x_{i+1}$, then again below each $A_k$ for which $x_k$ is Type 1 and above other $x_k$, then finally above $A_j$ and ends at $A_{j+1}$.}
\label{fig:arcs}
\end{figure}

For arcs $\gamma, \delta$ and corresponding half-twists $t_\gamma, t_\delta$, we have the relation $t_{t_\gamma(\delta)} = t_\gamma^{-1} t_\delta t_\gamma$. 
Thus applying a liftable element of $B_n$ preserves the Type of arcs; in particular, liftable arcs are taken to liftable arcs. We will later make use of this to reduce the above generating set.

\subsection{Conventions}\label{bg:conventions}

We will use the following conventions and notation.
\begin{itemize}
\item Throughout the paper, a branched covering will be denoted by $p: M \to N$. $N$ will usually be $S^3$ or $D^2$. All branched covers will be smooth, simple, and connected. All manifolds are oriented. As we work in dimensions 2 and 3, we do not distinguish between PL and smooth.
\item The monodromy representation of $p$ will be denoted $\rho^p: \pi_1(N - \text{branch set}) \to S_d$, where $S_d$ is the symmetric group on $d$ symbols.
\item A positive crossing in a braid has the left strand crossing over the right. A positive half-twist is anticlockwise.
\item The liftable subgroup with respect to $p$ is $L^p(B_n) < B_n$ (see Section \ref{bg:liftable}). The lifting map will be denoted $L: L^p(B_n) \to \MCG(\Sigma_g)$ with target the mapping class group of a genus $g$ surface.
\item For group elements $k$ and $h$, $[k]h$ will denote the conjugate $h^{-1}kh$.
\item The degree of the covering will be denoted by $d$.
\item The number of strands in the braid, or equivalently the number of branch points in $D^2$, will be denoted by $n$.
\item The genus of the Heegaard splitting induced by a branched cover, or equivalently the genus of a cover of $D^2$, will be denoted by $g$.
\end{itemize}

A diagram of the following form indicates a colored braid arising as a horizontal portion of a link in plat form. The rest of the link may be arbitrary.
The labels below the diagram specify which strands are involved.
The labels above denote the coloring at the top of each strand. This determines the coloring of the entire diagram (see Section \ref{bg:color}). When the diagram is standardly colored, the coloring labels will be omitted; the top of each strand should then be colored as in Figure \ref{fig:std}.

\bigskip
\begin{figure}
\labellist
\small
\pinlabel $(d\text{-}1\thinspace d)$ at -1 35
\pinlabel $(d\text{-}1\thinspace d)$ at 16 35
\pinlabel $(23)$ at 30 35
\pinlabel $(23)$ at 41 35
\pinlabel $(13)$ at 52 35
\pinlabel $(13)$ at 64 35
\pinlabel $(23)$ at 75 35
\pinlabel $(23)$ at 86 35
\pinlabel $(12)$ at 98 35
\pinlabel $(12)$ at 109 35
\pinlabel $(12)$ at 124 35
\pinlabel $(12)$ at 135 35

\pinlabel $0$ at 5 -5
\pinlabel $2d\text{-}10$ at 29 -4.5
\pinlabel $2d\text{-}4$ at 96 -4.5
\pinlabel $n\text{-}1$ at 136 -5
\tiny\pinlabel $\cdots$ at 17 -5
\pinlabel $\cdots$ at 116 -5
\endlabellist

\centering
\includegraphics[width=0.5\textwidth]{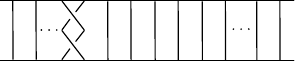}
\medskip
\caption{Example diagram of a portion of a colored link.}
\end{figure}

\section{Proofs}

\subsection{Outline}\label{pf:outline}

We outline the proof strategy for Theorem \ref{thm:moves}, which will allow us to prove Theorem \ref{thm:cn} and Corollary \ref{cor:lens} along the way. Given two links in plat form appearing as the branch locus of branched coverings by the same 3-manifold $M$, we will relate them as follows -- see also \cite{pie91} or \cite{apo03}.

\begin{enumerate}
\item Use isotopy to standardize the coloring of each link (Section \ref{pf:standard}).
\item Both links induce Heegaard splittings of $M$; these are stably equivalent. Carry out this stabilization on the cover by Markov stabilization on the link (Section \ref{pf:standard}).
\item Compute the liftable subgroup with respect to the standard coloring. Adapt the small generating set for $L^p(B_n)$ in Proposition \ref{prop:l-gens} to our setting (Section \ref{pf:gens}).

At this stage, we have the tools to prove Theorem \ref{thm:cn} and Corollary \ref{cor:lens}, but defer the proofs until Sections \ref{pf:cn} and \ref{pf:lens} respectively.

Examples of the generating set of $L^p(B_n)$ for small $d$ are provided in Section \ref{pf:example}.

\item Determine the image of each generator of $L^p(B_n)$ under the lifting homomorphism $L: L^p(B_n) \to \MCG(\Sigma_g)$ (Section \ref{pf:lift}).
\item The two induced Heegaard splittings are the same genus and equivalent, but possibly have different gluing homeomorphisms. Add a braid using isotopy and move C to carry out the homeomorphism extending over the handlebodies required to make the gluing maps the same (Section \ref{pf:gluing})
\item Now, since both braids lift to the same element of $\MCG(\Sigma_g)$, any difference between them must come from something in the kernel of $L$. It suffices to find a set of moves that allow the trivial braid to be replaced by any normal generator of this kernel. Compute a set of normal generators for $\ker L$ (Section \ref{pf:kernel}).
\item Lastly, show each normal generator of $\ker L$ can be added to the trivial braid using the moves of Theorem \ref{thm:moves}. Check these moves are covering moves in the first place (Section \ref{pf:moves}). 
\end{enumerate}

Theorem \ref{thm:sheets} is proven in Section \ref{pf:sheets}.

\subsection{Proof of Theorem \ref{thm:moves}}\label{pf:thm1}

\subsubsection{Normalization of induced Heegaard splittings}\label{pf:standard}

We first show the links' colorings can be standardized.

\begin{lemma}
The top (or equivalently, bottom) of a colored plat representing a connected covering can be made standard using only local isotopies.
\end{lemma}

We note that the following proof is very similar to that of Theorem A of \cite{mp01}, the difference being we can carry out elementary moves on the graphs using isotopy of plats.

\begin{proof}
We may view the coloring as a graph $\Gamma$ with vertices $\{1, \ldots, d\}$ and an edge between vertices $i, j$ for each arc that is colored $(i\ j)$. As the branched covering is connected, $\Gamma$ is a connected graph. 

We will induct on the degree of the covering. The base case, $d=3$, has been shown in \cite{pie91}; arcs may be reordered or their colors changed as in Figure \ref{fig:std-moves}. For the induction step, using the two isotopies in the figure, we do the following.

\bigskip
\begin{figure}
\labellist
\small\hair 2pt
\pinlabel $(i\ j)$ at 7 65
\pinlabel $(k\ l)$ at 36 65

\pinlabel $(i\ j)$ at 92 65
\pinlabel $(k\ l)$ at 122 65

\pinlabel $(i\ j)$ at 191 65
\pinlabel $(j\ k)$ at 221 65

\pinlabel $(i\ j)$ at 277 65
\pinlabel $(j\ k)$ at 307 65
\endlabellist

\centering
\includegraphics[width=0.7\textwidth]{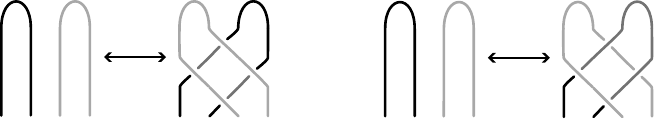}
\caption{Local isotopies reordering or changing coloring of plat arc. $i, j, k$, and $l$ need not be different.}
\label{fig:std-moves}
\end{figure}

\begin{enumerate}
\item Move all arcs with colorings involving $d$ to the left hand side. If some arc is colored $(d\text{-}1\ d)$, make it the leftmost arc. If not, choose an arc colored $(i\ d)$ for some $i$. As $\Gamma$ is connected, there is a path from vertex $i$ to $d\text{-}1$. Conjugate by the arc corresponding to each edge on the path to obtain an arc colored $(d\text{-}1\ d)$ and make that the leftmost arc. (If the path passes through edges $\{i, j_1, \ldots, j_k, d\text{-}1\}$ in order, conjugating will send $(i\ d) \mapsto (j_1\ d) \mapsto \cdots \mapsto (j_k\ d) \mapsto (d\text{-}1\ d)$.)
\item Next, we eliminate all other arcs colored $(d\text{-}1\ d)$. As $\Gamma$ is connected, either vertex $d\text{-}1$ or $d$ is joined to some other vertex $i \neq d\text{-}1, d$. Use that to conjugate all but one arc colored $(d\text{-}1\ d)$ to $(i\ d\text{-}1)$ or $(i\ d)$.
\item Finally, we eliminate all other arcs involving $d$, say colored $(i\ d)$, by conjugating with the arc $(d\text{-}1\ d)$ to obtain one colored $(i\ d\text{-}1)$. We are now reduced to the degree $d\text{-}1$ case.
\end{enumerate}
\end{proof}

Now we can assume that our link is colored as in Figure \ref{fig:std}. Next, carry out Heegaard stabilization in the cover by performing the required number of Markov stabilizations on the rightmost strand. 

\subsubsection{Generators for $L^p(B_n)$}\label{pf:gens}

We now compute an explicit generating set for standard $\rho^p$ from those given in Proposition \ref{prop:l-gens}.

\begin{theorem} \label{thm:gens}
For $\rho^p$ standard of degree $d$, a generating set for $L^p(B_n)$ is the union of the following:
\begin{align*}
\{ x_i & \mid i\geq 2d-4\text{ or } i < 2d-4 \text{ and even} \}
\\ \{ {x_i}^3 & \mid i\text{ odd and} < 2d-4 \}
\\ \{ {y_{i,j}}^2 & \mid i\text{ odd}, j-i\geq 3\text{ and odd}, j \leq 2d-4 \}
\\ \{ {z_{i,j}}^2 & \mid i\text{ odd}, j-i\geq 7\text{ and odd}, j \leq 2d-4 \}
\\ \{ w_{i,j} & \mid i,j\text{ even and }j \leq 2d-4 \}
\\ \{ s_1 \} &
\end{align*}
\end{theorem}

\begin{proof}
It is clear from Figure \ref{fig:arcs} that $x_i$ is Type 3 for $i$ odd and $< 2d-4$, and Type 1 otherwise. 
Then $y_{i,j}$, $z_{i,j}$, and $w_{i,j}$ for $i \geq 2d-4$ need not be included in the generating set, as they can be obtained from the minimal liftable powers of the $x_i$. Also note that if an arc shares an endpoint with a Type 1 arc, we may twist about the Type 1 arc to obtain another arc in $L^p(B_n)$. For this reason, we need not include $y_{i,j}$, $z_{i,j}$, and $w_{i,j}$ for $j > 2d-4$.

\begin{figure}
\centering

\begin{subfigure}[b]{0.32\textwidth}
\labellist
\small\hair 2pt
\pinlabel $j$ at 27 0
\pinlabel $j+1$ at 57 0
\pinlabel $j+2$ at 86 0
\endlabellist
\includegraphics[width=\textwidth]{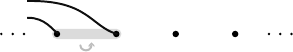}
\smallskip
\caption{Half-twist $w_{i,j}$ about the liftable arc $x_j$.}
\end{subfigure}
\hfill
\begin{subfigure}[b]{0.32\textwidth}
\labellist
\small\hair 2pt
\pinlabel $j+1$ at 57 0
\pinlabel $j+2$ at 86 0
\endlabellist
\includegraphics[width=\textwidth]{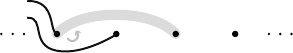}
\smallskip
\caption{Half-twist about the liftable arc $z_{j,j+2}$.}
\end{subfigure}
\hfill
\begin{subfigure}[b]{0.32\textwidth}
\labellist
\small\hair 2pt
\pinlabel $j+1$ at 57 0
\pinlabel $j+2$ at 86 0
\endlabellist
\includegraphics[width=\textwidth]{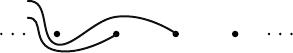}
\smallskip
\caption{The arc $w_{i,j+1}$.}
\bigskip
\end{subfigure}

\caption{Obtaining $w_{i,j+1}$ from $w_{i,j}$ using the other generators, for $j\geq 2d-4$.}
\end{figure}

By the above, we can also disregard $y_{i,j}$ for $i$ even, as they can be obtained by twisting $x_i$ or $y_{i,j}$ for $i$ odd about a Type 1 arc $x_i$. For $i$ odd, $y_{i,i+2}$ is Type 3 and obtained using liftable powers of the arcs $x_i$. The remaining cases, $y_{i,j}$ for $i$ odd, $j-i$ odd and $\geq 3$, are Type 2 and must be included in our generating set.

By definition, the possible arcs $w_{i,j}$ are those with $i,j$ even and $j-i$ even. To see that all $w_{i,j}$ are Type 1, note that in terms of the Artin generators, the half-twist about $w_{i,j}$, for $j \leq 2d-4$, is 
$$[x_j](x_{j-1} {x_{j-2}}^{-1})(x_{j-3} {x_{j-4}}^{-1})\cdots (x_{i+3} {x_{i+2}}^{-1}) (x_{i+1} x_i^2 x_{i+1}) ({x_{i+2}}^{-1} x_{i+3}) \cdots ({x_{j-2}}^{-1} x_{j-1})$$
As $x_i$ is Type 1 by definition, the generators occur in pairs with the same monodromy, $\rho^p(w_{i,j}) = \rho^p(x_j)$. Then $\rho^p(w_{i,j}(\alpha_j)) = \rho^p(x_j(\alpha)) = \rho^p(\alpha_{j+1}) = \rho^p(\alpha_j)$.

By \cite{ww12} Lemma 30, $z_{i,j}$ is Type 2 if and only if there is an arc $x_k$ of type 1 strictly between $A_i$ and $A_j$. As in the $y_{i,j}$ case, we may omit $i$ even, leaving $z_{i,j}$ with $i$ odd and $j-i$ odd. As $z_{i,i+3}$ and $z_{i,i+5}$ can be obtained from $y_{i,i+4}$ and $y_{i-1,i+6}$ respectively, using $w_{i+1,i+3}$ as shown in Figure \ref{fig:z from w}, we need only include $j-i \geq 7$.

\begin{figure}[t]
\centering

\begin{subfigure}[b]{0.45\textwidth}
\labellist
\small\hair 2pt
\pinlabel $i+1$ at 84 -5
\pinlabel $i+3$ at 140 -5
\pinlabel $i+4$ at 170 -5
\endlabellist
\includegraphics[width=\textwidth]{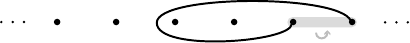}
\bigskip\bigskip
\caption{Half-twist $w_{i+1,i+3}$ about $x_{i+3}$. As both arcs are liftable, the resulting arc is also liftable.}
\end{subfigure}
\hfill
\begin{subfigure}[b]{0.45\textwidth}
\labellist
\small\hair 2pt
\pinlabel $i$ at 56 -5
\pinlabel $i+1$ at 84 -5
\pinlabel $i+3$ at 140 -5
\pinlabel $i+4$ at 170 -5
\endlabellist
\includegraphics[width=\textwidth]{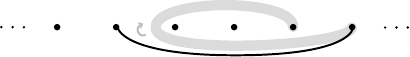}
\bigskip

\labellist
\small\hair 2pt
\pinlabel $i$ at 56 -5
\pinlabel $i+1$ at 84 -5
\pinlabel $i+3$ at 140 -5
\pinlabel $i+4$ at 170 -5
\endlabellist
\includegraphics[width=\textwidth]{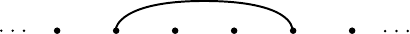}
\smallskip
\caption{Half-twist $y_{i,i+4}$ about the arc from (a) to obtain $z_{i,i+3}$. These will both be Type 2 arcs.}
\end{subfigure}

\caption{Obtaining $z_{i,i+3}$ from $y_{i,i+4}$. A similar method can be used to obtain $z_{i,i+5}$ using an arc like in (a) and its reflection in a vertical axis, which is also Type 1.}
\label{fig:z from w}
\end{figure}

As the excess (defined in Definition 15 of \cite{ww12}) of the standard coloring is $2d-4$, which is even, the only special arc that appears as a generator (when $d > 2$) is $s_1$, which has type 1. ($s_1$ is the analogue of $\delta$ in \cite{pie91}, or $\delta_6$ in \cite{apo03}).
\end{proof}

It is possible the generating set can be further reduced, but this suffices for this paper.

\subsubsection{Examples}\label{pf:example}
We now give examples, recovering the generators from \cite{bw85}, \cite{apo03}, and \cite{ww08} in degrees 3 and 4.

\begin{example}[$d=3$]
The generating set is $\{x_0, {x_1}^3, x_2, \ldots, x_{n-2}, w_{0,2}, s_1\}$, as no $y_{i,j}$ or $z_{i,j}$ fit the criteria of Proposition \ref{thm:gens}. 
\end{example}

The above generating set matches that of \cite{ww08}, but in fact $w_{0,2}$ is not necessary in our setting, as it is isotopic in $S^2$ to the arc in Figure \ref{fig:w-apo}, obtainable using $x_2, \ldots, x_{n-2}$. Thus we can reduce to a generating set matching that of \cite{bw85}.

\begin{figure}
\labellist
\small
\pinlabel $0$ at 2 -5
\pinlabel $1$ at 30 -5
\pinlabel $2$ at 58 -5
\pinlabel $3$ at 87 -5
\pinlabel $\cdots$ at 115 -5
\pinlabel $n\text{-}2$ at 144 -5
\pinlabel $n\text{-}1$ at 166 -5
\endlabellist
\centering
\includegraphics[width=0.45\textwidth]{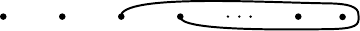}
\medskip
\caption{A liftable arc isotopic in $S^2$ to $w_{0,2}$ in degree 3.}
\label{fig:w-apo}
\end{figure}

\begin{example}[$d=4$]
The generating set is $\{x_0, {x_1}^3, x_2, {x_3}^3, x_4, \ldots, x_{n-2}, {y_{1,4}}^2, w_{0,2}, w_{0,4}, w_{2,4}, s_1\}$, and no $z_{i,j}$ fit the criteria of Proposition \ref{thm:gens}.
\end{example}

In the notation of \cite{ww08} (Theorem 3, Type 7), $y_{1,4}$ is $y$ and $s_1$ is $d$. $w$ can be obtained by half-twisting $w_{0,2}$ about $x_2$, $w_1$ by half-twisting $w_{2,4}$ by $x_4$, and $e$ can be obtained from $w_{0,4}$ using the other generators as shown in Figure \ref{fig:d4 gens}.

\begin{figure}
\centering

\begin{subfigure}[b]{0.32\textwidth}
\includegraphics[width=\textwidth]{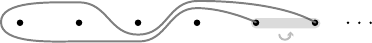}
\caption{Half-twist $w_{0,4}$ about $x_4$.}
\end{subfigure}
\hfill
\begin{subfigure}[b]{0.32\textwidth}
\includegraphics[width=\textwidth]{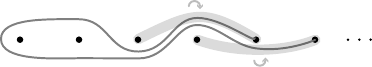}
\caption{Apply ${z_{2,4}}^{-3}$ and ${y_{3,5}}^3$.}
\end{subfigure}
\hfill
\begin{subfigure}[b]{0.32\textwidth}
\includegraphics[width=\textwidth]{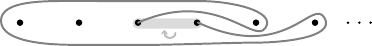}
\caption{Apply $x_2^{-1}$.}
\end{subfigure}

\bigskip

\begin{subfigure}[b]{0.32\textwidth}
\includegraphics[width=\textwidth]{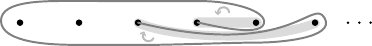}
\caption{Apply ${x_3}^3$ and ${y_{2,5}}^{-3}$.}
\end{subfigure}
\quad\quad
\begin{subfigure}[b]{0.32\textwidth}
\includegraphics[width=\textwidth]{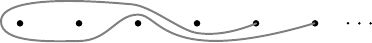}
\caption{The arc $e$.}
\end{subfigure}

\caption{Obtaining $e$ as a conjugate of $w_{0,4}$.}
\label{fig:d4 gens}
\end{figure}

Apostolakis \cite{apo03} gives generators for the quotient of $L^p(B_n)$ obtained by equating braids that differ by a sequence of moves C and N. In this quotient, our generating set becomes $\{x_0, x_2, x_4, \ldots, x_{n-2}, s_1\}$, matching that of \cite{apo03}. We can again make use of the spherical relation to express $w_{0,2}, w_{0,4}$, and $w_{2,4}$ in terms of the other generators as shown in Figures \ref{fig:w-iso-2} and \ref{fig:w-iso-4}. Note that as the procedure to do so relies on $z_{i,j}^2$, which is only trivial up to move N, we cannot remove these from the generating set of the liftable subgroup itself.

\begin{figure}
\centering

\begin{subfigure}[b]{0.6\textwidth}
\labellist
\footnotesize
\pinlabel $0$ at 2 -5
\pinlabel $1$ at 25 -5
\pinlabel $i$ at 185 -5
\pinlabel $i\text{+}2$ at 245 -5.5
\pinlabel $n\text{-}2$ at 377 -5
\pinlabel $n\text{-}1$ at 401 -5
\endlabellist
\includegraphics[width=\textwidth]{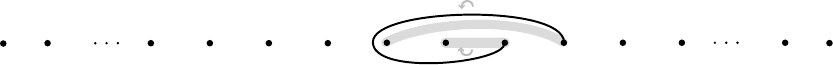}
\vspace{-5pt}
\caption*{Half-twist $w_{i,i+2}$ about ${z_{i,i+3}}^3$ and ${x_{i+1}}^{-3}$.}

\bigskip
\labellist
\footnotesize
\pinlabel $0$ at 2 -5
\pinlabel $1$ at 25 -5
\pinlabel $i$ at 185 -5
\pinlabel $i\text{+}1$ at 215 -5.5
\pinlabel $n\text{-}2$ at 377 -5
\pinlabel $n\text{-}1$ at 401 -5
\endlabellist
\includegraphics[width=\textwidth]{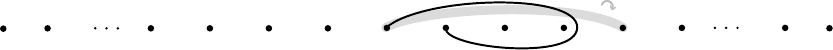}
\vspace{-5pt}
\caption*{Apply $z_{i,j}^2$ for $i+4 \leq j \leq n\text{-}1$.}

\bigskip
\labellist
\footnotesize
\pinlabel $0$ at 2 -5
\pinlabel $1$ at 25 -5
\pinlabel $i$ at 185 -5
\pinlabel $i\text{+}1$ at 215 -5.5
\pinlabel $n\text{-}2$ at 377 -5
\pinlabel $n\text{-}1$ at 401 -5
\endlabellist
\includegraphics[width=\textwidth]{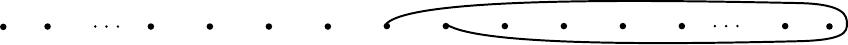}
\vspace{-5pt}
\caption*{Perform isotopy around $S^2$.}

\bigskip
\labellist
\footnotesize
\pinlabel $0$ at 7 -5
\pinlabel $1$ at 31 -5
\pinlabel $i$ at 192 -5
\pinlabel $i\text{+}1$ at 225 -5.5
\pinlabel $n\text{-}2$ at 383 -5
\pinlabel $n\text{-}1$ at 410 -5
\endlabellist
\includegraphics[width=\textwidth]{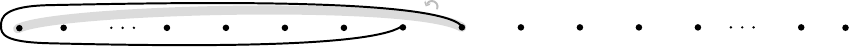}
\vspace{-5pt}
\caption*{Apply $z_{j,i+1}^2$ for $0 \leq j \leq i-3$.}

\bigskip
\labellist
\footnotesize
\pinlabel $0$ at 2 -5
\pinlabel $1$ at 25 -5
\pinlabel $i\text{-}2$ at 128 -5.5
\pinlabel $i$ at 185 -5
\pinlabel $n\text{-}2$ at 377 -5
\pinlabel $n\text{-}1$ at 401 -5
\endlabellist
\includegraphics[width=\textwidth]{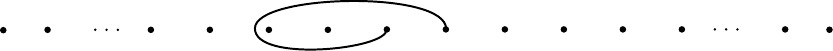}
\vspace{-5pt}
\caption*{The arc $w_{i-2,i}$.}
\medskip
\caption{Reducing $w_{i,i+2}$ to $w_{i-2,i}$.}
\end{subfigure}
\hfill
\begin{subfigure}[b]{0.325\textwidth}

\labellist
\footnotesize
\pinlabel $0$ at 5 -5
\pinlabel $2$ at 64 -5
\pinlabel $n\text{-}2$ at 197 -5
\pinlabel $n\text{-}1$ at 222 -5
\endlabellist
\includegraphics[width=\textwidth]{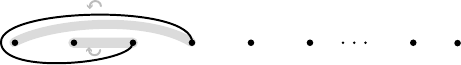}
\vspace{-5pt}
\caption*{Half-twist $w_{0,2}$ about ${z_{0,3}}^3$ and ${x_1}^{-3}$.}

\bigskip
\labellist
\footnotesize
\pinlabel $0$ at 4 -5
\pinlabel $1$ at 32 -5
\pinlabel $n\text{-}2$ at 195 -5
\pinlabel $n\text{-}1$ at 220 -5
\endlabellist
\includegraphics[width=\textwidth]{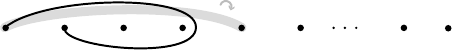}
\vspace{-5pt}
\caption*{Apply $z_{0,4}^2$.}

\bigskip
\labellist
\footnotesize
\pinlabel $0$ at 4 -5
\pinlabel $1$ at 32 -5
\pinlabel $n\text{-}2$ at 195 -5
\pinlabel $n\text{-}1$ at 220 -5
\endlabellist
\includegraphics[width=\textwidth]{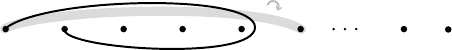}
\vspace{-5pt}
\caption*{Apply $z_{0,j}^2$ for $5 \leq j \leq n\text{-}1$.}

\bigskip
\labellist
\footnotesize
\pinlabel $0$ at 3 -5
\pinlabel $1$ at 31 -5
\pinlabel $n\text{-}2$ at 193 -5
\pinlabel $n\text{-}1$ at 218 -5
\endlabellist
\includegraphics[width=\textwidth]{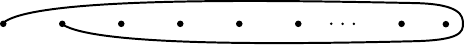}
\vspace{-5pt}
\caption*{Perform isotopy around $S^2$.}

\bigskip
\labellist
\footnotesize
\pinlabel $0$ at 3 -15
\pinlabel $1$ at 31 -15
\pinlabel $n\text{-}2$ at 193 -15
\pinlabel $n\text{-}1$ at 218 -15
\endlabellist
\includegraphics[width=\textwidth]{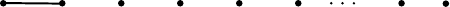}
\vspace{-5pt}
\caption*{The arc $x_0$.}
\medskip
\medskip
\caption{Reducing $w_{0,2}$ to $x_0$.}
\label{fig:w-iso-2b}
\end{subfigure}

\caption{Obtaining $w_{i,i+2}$ from $x_j$ for $j$ even, C, and N. All arcs are standardly colored.}
\label{fig:w-iso-2}
\end{figure}

\bigskip

\begin{figure}
\centering

\begin{subfigure}[b]{0.517\textwidth}
\labellist
\footnotesize
\pinlabel $0$ at 2 -5
\pinlabel $1$ at 25 -5
\pinlabel $i$ at 128 -5
\pinlabel $i\text{+}4$ at 245 -5.5
\pinlabel $n\text{-}2$ at 327 -5
\pinlabel $n\text{-}1$ at 352 -5
\endlabellist
\includegraphics[width=\textwidth]{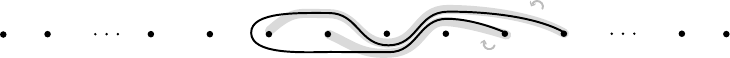}
\vspace{-5pt}
\caption*{Half-twist $w_{i,i+4}$ thrice about the Type 3 arcs shown.}
\bigskip

\bigskip
\labellist
\footnotesize
\pinlabel $0$ at 2 -5
\pinlabel $1$ at 25 -5
\pinlabel $i$ at 128 -5
\pinlabel $i\text{+}1$ at 159 -5.5
\pinlabel $n\text{-}2$ at 327 -5
\pinlabel $n\text{-}1$ at 352 -5
\endlabellist
\includegraphics[width=\textwidth]{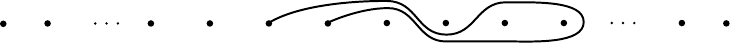}
\vspace{-5pt}
\caption*{Perform isotopy around $S^2$.}

\bigskip
\labellist
\footnotesize
\pinlabel $0$ at 2 -5
\pinlabel $1$ at 25 -5
\pinlabel $i$ at 128 -5
\pinlabel $i\text{+}1$ at 159 -6
\pinlabel $n\text{-}2$ at 327 -5.5
\pinlabel $n\text{-}1$ at 352 -5
\endlabellist
\includegraphics[width=\textwidth]{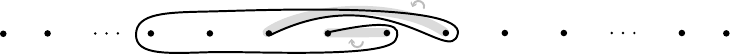}
\vspace{-5pt}
\caption*{Apply ${z_{i,i+3}}^3$ and $x_{i+1}^{-3}$.}

\bigskip
\labellist
\footnotesize
\pinlabel $0$ at 2 -5
\pinlabel $1$ at 25 -5
\pinlabel $i\text{-}2$ at 69 -5.5
\pinlabel $i\text{+}2$ at 189 -5.5
\pinlabel $n\text{-}2$ at 327 -5
\pinlabel $n\text{-}1$ at 352 -5
\endlabellist
\includegraphics[width=\textwidth]{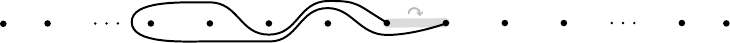}
\vspace{-5pt}
\caption*{Apply ${x_{i+2}}^{-1}$.}

\bigskip
\labellist
\footnotesize
\pinlabel $0$ at 2 -5
\pinlabel $1$ at 25 -5
\pinlabel $i\text{-}2$ at 69 -5.5
\pinlabel $i\text{+}2$ at 189 -5
\pinlabel $n\text{-}2$ at 327 -5
\pinlabel $n\text{-}1$ at 352 -5
\endlabellist
\includegraphics[width=\textwidth]{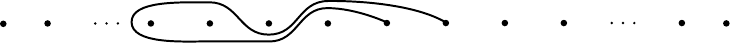}
\vspace{-5pt}
\caption*{The arc $w_{i-2,i+2}$.}
\medskip
\caption{Reducing $w_{i,i+4}$ to $w_{i-2,i+2}$.}
\end{subfigure}
\hfill
\begin{subfigure}[b]{0.413\textwidth}
\labellist
\footnotesize
\pinlabel $0$ at 8 -5
\pinlabel $4$ at 123 -5
\pinlabel $n\text{-}2$ at 256 -5
\pinlabel $n\text{-}1$ at 281 -5
\endlabellist
\includegraphics[width=\textwidth]{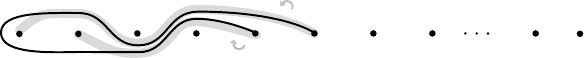}
\vspace{-5pt}
\caption*{Half-twist $w_{0,4}$ thrice about the Type 3 arcs shown.}

\bigskip
\labellist
\footnotesize
\pinlabel $0$ at 4 -5
\pinlabel $1$ at 32 -5
\pinlabel $n\text{-}2$ at 250 -5
\pinlabel $n\text{-}1$ at 275 -5
\endlabellist
\includegraphics[width=\textwidth]{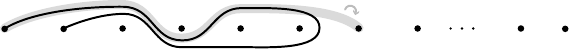}
\vspace{-5pt}
\caption*{Use moves N as in Figure \ref{fig:w-iso-2b}.}

\bigskip
\labellist
\footnotesize
\pinlabel $0$ at 4 -5
\pinlabel $1$ at 32 -5
\pinlabel $n\text{-}2$ at 250 -5.5
\pinlabel $n\text{-}1$ at 275 -5.5
\endlabellist
\includegraphics[width=\textwidth]{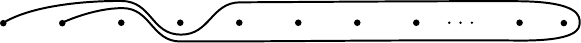}
\vspace{-5pt}
\caption*{Perform isotopy around $S^2$.}

\medskip
\labellist
\footnotesize
\pinlabel $0$ at 10 -5
\pinlabel $1$ at 38 -5
\pinlabel $n\text{-}2$ at 256 -5
\pinlabel $n\text{-}1$ at 281 -5
\endlabellist
\includegraphics[width=\textwidth]{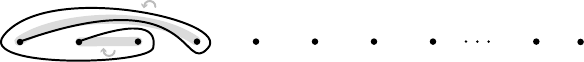}
\vspace{-5pt}
\caption*{Apply ${z_{0,3}}^{3}$ and ${x_1}^{-3}$.}

\bigskip
\labellist
\footnotesize
\pinlabel $0$ at 2 -13
\pinlabel $1$ at 30 -13
\pinlabel $2$ at 59 -13
\pinlabel $3$ at 87 -13
\pinlabel $n\text{-}2$ at 249 -13
\pinlabel $n\text{-}1$ at 273 -13
\endlabellist
\includegraphics[width=\textwidth]{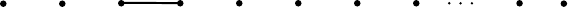}
\smallskip
\caption*{The arc $x_2$.}
\medskip
\caption{Reducing $w_{0,4}$ to $x_2$.}
\end{subfigure}

\caption{Obtaining $w_{i,i+4}$ from $x_j$ for $j$ even, C, and N. All arcs are standardly colored.}
\label{fig:w-iso-4}
\end{figure}

\begin{example}[$d=5$]
The generating set is $$\{x_0, {x_1}^3, x_2, {x_3}^3, x_4, {x_5}^3, x_6, \ldots, x_{n-2}, {y_{1,4}}^2, {y_{3,6}}^2, {y_{1,6}}^2, w_{0,2}, w_{2,4}, w_{4,6}, w_{0,4}, w_{2,6}, w_{0,6}, s_1\}$$
\end{example}

In degree $d \geq 5$, there are $n-1 + \frac{(d-3)(d-2)}{2} + \frac{(d-5)(d-4)}{2} + \frac{(d-2)(d-1)}{2} + 1 = \frac{3d^2-17d+28}{2}+n$ generators.

\subsubsection{The lifting homomorphism}\label{pf:lift}

Now we can compute the lifting homomorphism. Let $a_i, b_i, \delta$ denote the generators of $\MCG(\Sigma_g)$ as per \cite{waj83}, where our $\delta$ is $d$ in the notation of \cite{waj83}, and $x_i$ the generators of $B_n$. Then

\begin{theorem}
The lifting homomorphism is given by
$$ L(\beta) = 
\begin{cases}
id & \beta = x_{2i},\ i<d-2,$ even$
\\ id & \beta = {x_{2i+1}}^3,\ i<d-2,$ odd$
\\ a_{2(i-d+2)} & \beta = x_{2i},\ i\geq d-2
\\ b_{2(i-d+2)+1} & \beta = x_{2i+1},\ i\geq d-2
\\ \delta & \beta = s_1
\end{cases}
$$
\end{theorem}

\begin{proof}
The image of the half-twist about an arc in $S^2$ is determined by the lift of that arc under the branched covering map. As the branched covering is simple, the lift of an arc will consist of the disjoint union of arcs and up to one closed curve in $\Sigma_g$. The half-twist then lifts to a Dehn twist about any closed components and a half-twist about any arcs. Dehn twists about nullhomotopic curves and any half-twist about an arc in $\Sigma_g$ are isotopic to the identity mapping class. We can directly compute the arc or curve to which an arc in $S^2$ lifts using a cut-and-paste technique, as shown in Figure \ref{fig:cutpaste}.
\end{proof}

\begin{figure}
\labellist
\small\hair 2pt
\pinlabel $a_0$ at 18 50
\pinlabel $a_1$ at 89 50
\pinlabel $a_2$ at 160 50
\pinlabel $a_{g-1}$ at 231 50
\pinlabel $a_g$ at 302 16
\pinlabel $\delta$ at 111 12
\pinlabel $b_1$ at 56 56
\pinlabel $b_2$ at 127 56
\pinlabel $b_g$ at 269 56
\endlabellist

\centering
\includegraphics[width=0.7\textwidth]{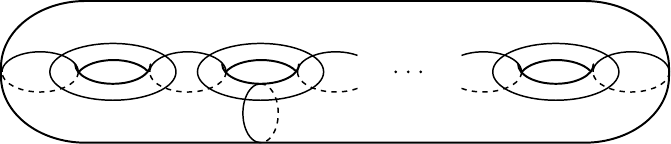}
\caption{Wajnryb's generating set for $\MCG(\Sigma_g)$ consists of Dehn twists about curves $\{a_0,\ldots,a_g,b_1,\ldots,b_g,\delta\}$.}
\label{fig:mcg gens}
\end{figure}

\bigskip\bigskip

\begin{figure}[t]
\labellist
\large\pinlabel $S^2$ at -45 25
\small\pinlabel $1$ at -45 140
\pinlabel $2$ at -45 212
\pinlabel $3$ at -45 284
\pinlabel $d-2$ at -45 356
\pinlabel $d-1$ at -45 428
\pinlabel $d$ at -45 500
\pinlabel {Sheet \#} at -45 545
\pinlabel $p$ at 250 82
\pinlabel \textcolor{nicered}{$w_{0,2}$} at 55 57
\pinlabel \textcolor{niceblue}{$s_1$} at 280 57
\pinlabel \textcolor{nicered}{$p^{-1}(w_{0,2})$} at 55 535
\pinlabel \textcolor{niceblue}{$p^{-1}(s_1)$} at 280 535
\endlabellist
\centering
\includegraphics[width=0.7\textwidth]{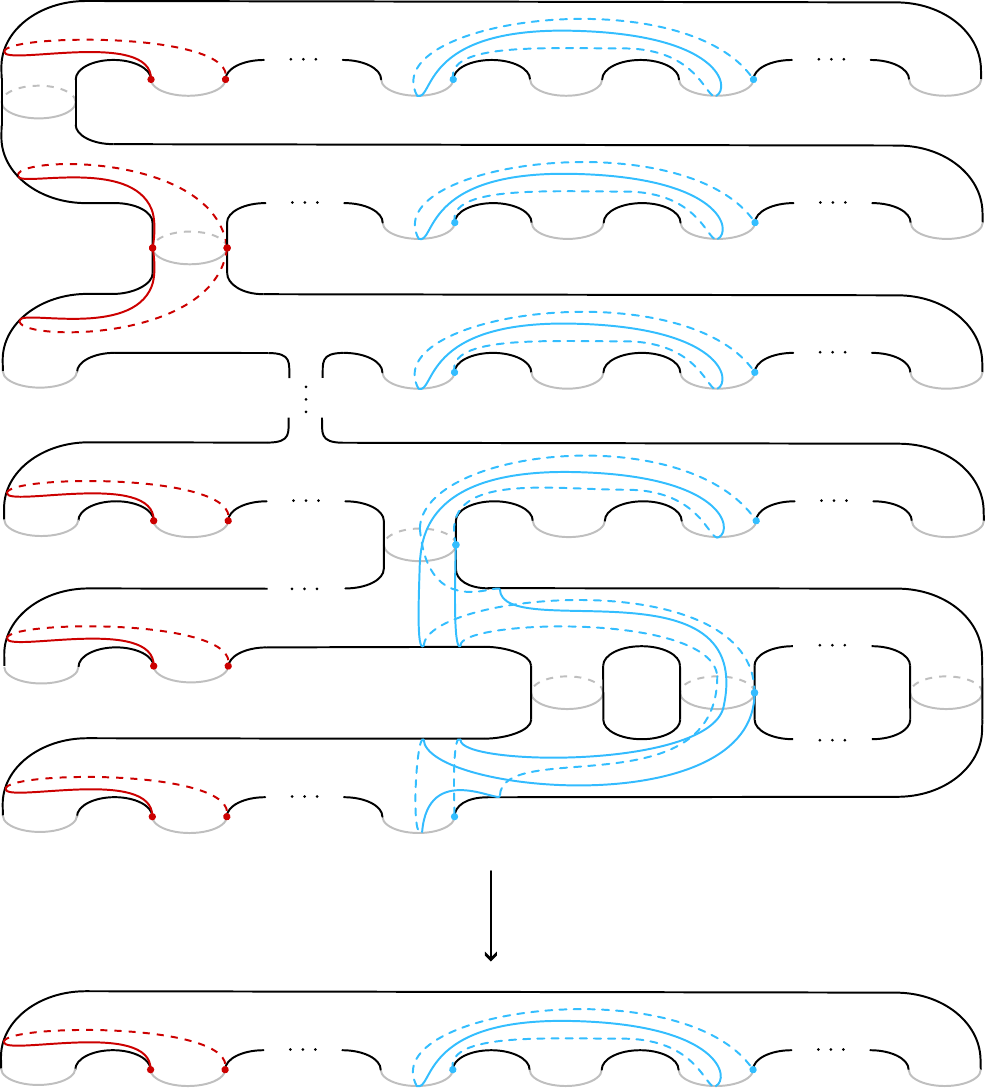}
\caption{A cut-and-paste diagram showing the construction of the standard degree $d$ cover of $S^2$ and the lifts of arcs $w_{0,2}$ and $s_1$. As both arcs are liftable, their lifts contain a unique closed loop. The loop in $p^{-1}(w_{0,2})$ is homotopically trivial, so $w_{0,2}$ is in the kernel of $L$. The loop in $p^{-1}(s_1)$ is isotopic to the loop $\delta$ in Figure \ref{fig:mcg gens}.}
\label{fig:cutpaste}
\end{figure}

\subsubsection{Equivalence of Heegaard splittings}\label{pf:gluing}

In order to turn equivalent Heegaard splittings into Heegaard splittings with the same gluing map, we must apply a homeomorphism to one of the handlebodies involved. As every such homeomorphism is determined by a homeomorphism of its boundary surface, we need only find some liftable braid that lifts to each element of $\MCG(\Sigma_g)$ that extends over a handlebody. Suzuki provides a generating set for this subgroup of $\MCG(\Sigma_g)$ in \cite{suz77}. Piergallini exhibits braids that lift to each generator in Section 2 of \cite{pie91}, which require only move C and isotopy to add to the trivial braid. By stabilizing these braids on the left with trivial strands colored as per the standard coloring, we obtain the desired braids for higher degree covers.

\subsubsection{Normal generators for $\ker L$}\label{pf:kernel}

$\ker L$ is normally generated by the generators in Theorem \ref{thm:gens} that are sent to the identity, and the preimage of the relations in $\MCG(\Sigma_g)$. 

\begin{theorem}
A normal generating set for the kernel of the lifting homomorphism is the union of:
\begin{align*}
\{ x_i & \mid i < 2d-4 \text{ and even} \}
\\ \{ {x_i}^3 & \mid i\text{ odd and} < 2d-4 \}
\\ \{ {y_{i,j}}^2 & \mid i\text{ odd}, j-i\text{ odd and }\geq 3, j \leq 2d-4 \}
\\ \{ {z_{i,j}}^2 & \mid i\text{ odd}, j-i\text{ odd and }\geq 5, j \leq 2d-4 \}
\\ \{ w_{i,j} & \mid i,j\text{ even and }j \leq 2d-4 \}
\\ \{ B, &\ d_{n-3} {x_{n-2}}^{-1} \}\ \text{ (in the notation of Theorem 6.1 of \cite{bw85})}
\end{align*}

In terms of the Artin generators, $B = (x_{2d-4}x_{2d-3}x_{2d-2})^4[u^{-1}]y^{-1}u^{-1}$, where
\begin{align*}
y & = x_{2d-1} \cdots x_{2d-3}{x_{2d-4}}^2x_{2d-3} \cdots x_{2d-1},
\\ u & = [x_{2d-2}](x_{2d-3}x_{2d-4}{x_{2d-5}}^2x_{2d-4}{x_{2d-3}}^2x_{2d-4}x_{2d-5})\text{, and}
\\ d_{n-3} & = [x_{n-4}]x_{n-3}x_{n-4} \cdots x_{2d-4}{x_{2d-5}}^2 x_{2d-4} \cdots x_{n-4}{x_{n-3}}^2 x_{n-4} \cdots x_{2d-4}x_{2d-5}
\end{align*}
\end{theorem}

Liftable braids involving only strands 0 through $2d-5$ will be in the kernel as they lift to twists around homotopically trivial curves (see Figure \ref{fig:cutpaste}). The generators $B$ and $d_{n-3}{x_{n-2}}^{-1}$ come from the 3-chain relation and lantern relation respectively. In the notation of \cite{pie91}, $d_{n-3}{x_{n-2}}^{-1}$ is $y_{g}^{-1} {x_{2g+2}}^{-1}y_g x_{2g+2}$.

\subsubsection{Adding normal generators for $\ker L$ to the trivial braid using moves}\label{pf:moves}

We first confirm that the moves of Theorem \ref{thm:moves} are covering moves by showing that the braid appearing on the right-hand side lies in $\ker L$, and so lifts to the identity mapping class. 
By the same arguments as in \cite{pie91}, moves II$_i$ and IV are in $\ker L$.
Move III$_{i,j}$ is in $\ker L$ as it can be obtained from $w_{i,j}$ using moves C and II.

Next, we show each braid in $\ker L$ can be obtained from the trivial braid by some sequence of covering moves. Move II$_i$ and C can be used to realize $x_{2i}$ and $x_{2i+1}$ respectively. Move N can be used to realize ${y_{i,j}}^2$ and ${z_{i,j}}^2$ as in Figure \ref{fig:yz-moves}. Figure \ref{fig:w-moves} shows how to realize $w_{i,j}$. Lastly, $B$ and $d_{n-3}{x_{n-2}}^{-1}$ can be realized as in Figures 19-21 of \cite{pie91}, as only strands colored $(1\ 2), (2\ 3),$ or $(1\ 3)$ are involved. The procedure in Figure 20 of \cite{pie91} requires additional N moves between stages (c) and (d) for higher-degree covers.

\begin{figure}
\centering

\begin{subfigure}[b]{0.3\textwidth}
\labellist
\small\pinlabel $i$ at 28 -8
\pinlabel $i\text{+}1$ at 42 -8
\pinlabel $\cdots$ at 65 -8
\pinlabel $j\text{-}1$ at 85 -8
\pinlabel $j$ at 98 -8
\endlabellist
\includegraphics[width=\textwidth]{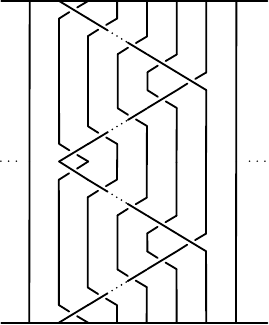}
\medskip
\caption{The braid ${y_{i,j}}^2$. The strands are colored as per the standard coloring in Fig. \ref{fig:std}.}
\end{subfigure}
\hfill
\begin{subfigure}[b]{0.3\textwidth}
\labellist
\small\pinlabel $i$ at 28 -8
\pinlabel $i\text{+}1$ at 42 -8
\pinlabel $\cdots$ at 65 -8
\pinlabel $j\text{-}1$ at 85 -8
\pinlabel $j$ at 98 -8
\endlabellist
\includegraphics[width=\textwidth]{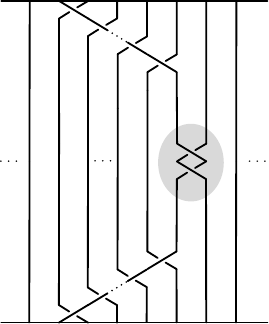}
\medskip
\caption{Isotopy from (a).}
\bigskip\bigskip
\end{subfigure}
\hfill
\begin{subfigure}[b]{0.3\textwidth}
\labellist
\small\pinlabel $i$ at 28 -8
\pinlabel $i\text{+}1$ at 42 -8
\pinlabel $\cdots$ at 65 -8
\pinlabel $j\text{-}1$ at 85 -8
\pinlabel $j$ at 98 -8
\endlabellist
\includegraphics[width=\textwidth]{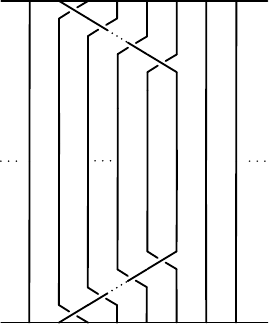}
\medskip
\caption{Move N in the shaded region of (b). This braid is isotopic to the trivial braid.}
\end{subfigure}

\caption{Undoing ${y_{i,j}}^2$ using N. The procedure for ${z_{i,j}}^2$ is the same.}
\label{fig:yz-moves}
\end{figure}

\bigskip\bigskip

\begin{figure}
\centering

\begin{subfigure}[b]{0.23\textwidth}
\labellist
\small\pinlabel $i$ at 14 -7.5
\pinlabel $i\text{+}1$ at 30 -8
\pinlabel $\cdots$ at 57 -8
\pinlabel $j$ at 80 -8
\pinlabel $j\text{+}1$ at 96 -8
\endlabellist
\includegraphics[width=\textwidth]{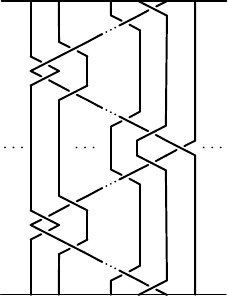}
\medskip
\caption{The braid $w_{i,j}$.}
\bigskip
\end{subfigure}
\hfill
\begin{subfigure}[b]{0.23\textwidth}
\labellist
\small\pinlabel $i$ at 14 -7.5
\pinlabel $i\text{+}1$ at 30 -8
\pinlabel $\cdots$ at 57 -8
\pinlabel $j$ at 80 -8
\pinlabel $j\text{+}1$ at 96 -8
\endlabellist
\includegraphics[width=\textwidth]{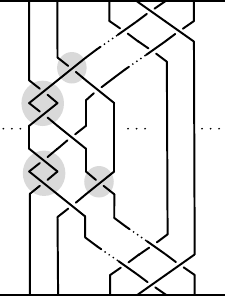}
\medskip
\caption{Isotopy from (a).}
\bigskip
\end{subfigure}
\hfill
\begin{subfigure}[b]{0.23\textwidth}
\labellist
\small\pinlabel $i$ at 14 -7.5
\pinlabel $i\text{+}1$ at 30 -8
\pinlabel $\cdots$ at 57 -8
\pinlabel $j$ at 80 -8
\pinlabel $j\text{+}1$ at 96 -8
\endlabellist
\includegraphics[width=\textwidth]{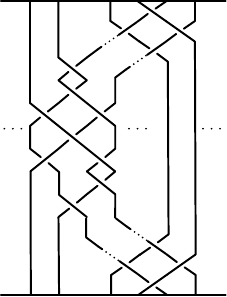}
\medskip
\caption{Move C in each shaded region in (b).}
\end{subfigure}
\hfill
\begin{subfigure}[b]{0.23\textwidth}
\labellist
\small\pinlabel $i$ at 14 -7.5
\pinlabel $i\text{+}1$ at 30 -8
\pinlabel $\cdots$ at 57 -8
\pinlabel $j$ at 80 -8
\pinlabel $j\text{+}1$ at 96 -8
\endlabellist
\includegraphics[width=\textwidth]{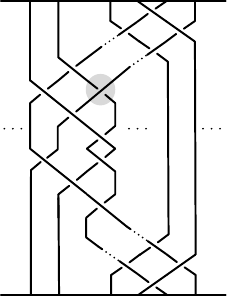}
\medskip
\caption{Isotopy from (c).}
\bigskip
\end{subfigure}
\end{figure}

\begin{figure}\ContinuedFloat
\begin{subfigure}[b]{0.23\textwidth}
\labellist
\small\pinlabel $i$ at 14 -7.5
\pinlabel $i\text{+}1$ at 30 -8
\pinlabel $\cdots$ at 57 -8
\pinlabel $j$ at 80 -8
\pinlabel $j\text{+}1$ at 96 -8
\endlabellist
\includegraphics[width=\textwidth]{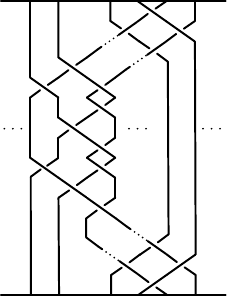}
\medskip
\caption{Move C in the shaded region of (d).}
\end{subfigure}
\hfill
\begin{subfigure}[b]{0.23\textwidth}
\labellist
\small\pinlabel $i$ at 14 -7.5
\pinlabel $i\text{+}1$ at 30 -8
\pinlabel $\cdots$ at 57 -8
\pinlabel $j$ at 80 -8
\pinlabel $j\text{+}1$ at 96 -8
\endlabellist
\includegraphics[width=\textwidth]{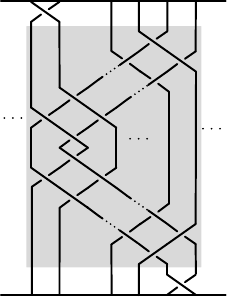}
\medskip
\caption{Isotopy from (e).}
\bigskip
\end{subfigure}
\hfill
\begin{subfigure}[b]{0.23\textwidth}
\labellist
\small\pinlabel $i$ at 14 -7.5
\pinlabel $i\text{+}1$ at 30 -8
\pinlabel $\cdots$ at 57 -8
\pinlabel $j$ at 80 -8
\pinlabel $j\text{+}1$ at 96 -8
\endlabellist
\includegraphics[width=\textwidth]{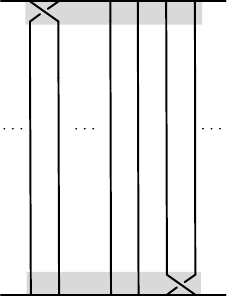}
\medskip
\caption{Move III$_{i,j}$ in the shaded region of (f).}
\end{subfigure}
\hfill
\begin{subfigure}[b]{0.23\textwidth}
\labellist
\small\pinlabel $i$ at 14 -7.5
\pinlabel $i\text{+}1$ at 30 -8
\pinlabel $\cdots$ at 57 -8
\pinlabel $j$ at 80 -8
\pinlabel $j\text{+}1$ at 96 -8
\endlabellist
\includegraphics[width=\textwidth]{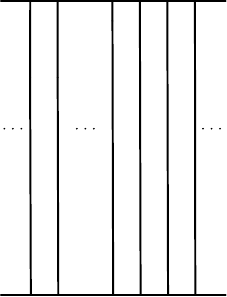}
\medskip
\caption{Moves II$_i$ and II$_j$ in the shaded regions of (g).}
\end{subfigure}

\caption{Undoing $w_{i,j}$ with C, II$_i$, II$_j$, and III$_{i,j}$. All strands are colored standardly.}
\label{fig:w-moves}
\end{figure}

\subsection{Proof of Theorem \ref{thm:cn}}\label{pf:cn}

In this section, we show that up to stabilization to the same degree at least 4, only C and N are needed to relate links with the same cover.

We first show the following lemma, which will also be needed in the proof of Corollary \ref{cor:lens}.

\begin{lemma} Up to C, N, and isotopy, each $L_i$ is the stabilization of the branch set of a 3-fold cover. \label{lem:stab}\end{lemma}
\begin{proof}
It suffices to show that up to C, N, and isotopy, no instances of the Artin generators $x_i$ for $i$ odd and less than $2d-6$ appear. 

First, stabilize colored links $L_1$ and $L_2$ to the same degree $d$ and write them as the plat closures of braids in standard form as in Section \ref{pf:standard}. These braids will each be some word in the generators of $L^p(B_n)$ computed in Section \ref{pf:gens}.

Instances of ${x_i}^3$ for $i$ odd and less than $2d-4$ can be removed using move C. Instances of ${y_{i,j}}^2$ and ${z_{i,j}}^2$ can be removed using move N. The last appearance of any $w_{i,j}$ in the word can be removed by link isotopy.

This process can be repeated until we are left with a word in only $x_0, x_2, \ldots, x_{2d-8}, x_{2d-6}, x_{2d-5}, \ldots, x_{n-2}$, the plat closure of which describes the same manifold as the original braid. This plat closure is the split union of $d-3$ unlinks colored $(d\text{-}1\ d)$ through $(3\ 4)$ and some 3-colored link, and so is a stabilization to degree $d$ of the branch set of a degree 3 cover.
\end{proof}

The proof of Theorem \ref{thm:cn} is completed by noting that by \cite{pie95}, the branch sets of any two 3-fold simple branched covers by the same 3-manifold can be related by only C and N after one stabilization.

\subsection{Proof of Theorem \ref{thm:sheets}}\label{pf:sheets}
In this section, we show that after $d-2$ stabilizations, each of the moves of Theorem \ref{thm:moves} can be carried out using only C and N.  Figure \ref{fig:useful} shows a useful move in the proof of this theorem.

\begin{figure}
\centering
\labellist
\large\pinlabel {$2\times$C} at 76 55
\pinlabel $\simeq$ at 174 55

\small\pinlabel $(i\ j)$ at 8 95
\pinlabel $(j\ k)$ at 36 95
\pinlabel $(i\ l)$ at 8 -8
\pinlabel $(l\ k)$ at 36 -8
\pinlabel $(j\ l)$ at -20 43

\pinlabel $(i\ j)$ at 108 95
\pinlabel $(j\ k)$ at 136 95
\pinlabel $(i\ l)$ at 108 -8
\pinlabel $(l\ k)$ at 136 -8

\pinlabel $(i\ j)$ at 208 95
\pinlabel $(j\ k)$ at 236 95
\pinlabel $(i\ l)$ at 208 -8
\pinlabel $(l\ k)$ at 236 -8
\endlabellist
\bigskip
\includegraphics[width=0.45\textwidth]{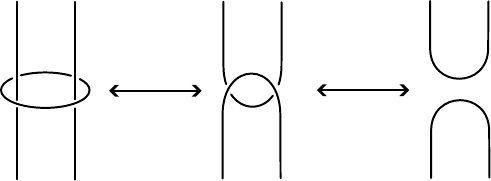}
\bigskip
\caption{An auxiliary move. The labels $i$ and $k$ may coincide.}
\label{fig:useful}
\end{figure}

\subsubsection{Move II}

II$_0$ can be obtained starting from the trivial braid using one stabilizing unknot colored $(d\ d\text{+}1)$, as in \cite{pie95}. This is shown in Figure \ref{fig:sheet ii0}. Undoing each II$_i$ for $i < d\text{-}3$ requires an additional stabilization colored $(d\text{-}j\ d\text{+}1\text{+}j)$ for each $j < i$ and is shown in Figure \ref{fig:sheet iii}. This can be thought of as a local version of destabilization, as the strands colored $(d\text{-}j\ d\text{-}j\text{-}1)$ can be isotoped away from the rest of the diagram. II$_{d-3}$ can be undone without an additional stabilization, as in \cite{apo03}, and shown in Figure \ref{fig:sheet iid}. In total, for $d \geq 4$, $d-3$ extra sheets are used.

\begin{figure}
\centering
\begin{subfigure}[b]{0.384\textwidth}
\labellist
\small\pinlabel $(d\ d\text{+}1)$ at 165 35
\pinlabel $0$ at 5 -10
\pinlabel $1$ at 20 -10
\pinlabel $2$ at 35 -10
\pinlabel $3$ at 50 -10
\pinlabel $\cdots$ at 85 -10
\pinlabel $n\text{-}2$ at 118 -10
\pinlabel $n\text{-}1$ at 135 -10
\endlabellist
\includegraphics[width=\textwidth]{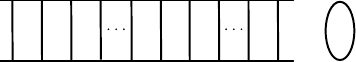}
\medskip
\caption{The trivial braid, stabilized once.}
\end{subfigure}
\quad\qquad
\begin{subfigure}[b]{0.32\textwidth}
\labellist
\small\pinlabel $(d\ d\text{+}1)$ at 162 15
\pinlabel $0$ at 5 -10
\pinlabel $1$ at 20 -10
\pinlabel $2$ at 35 -10
\pinlabel $3$ at 50 -10
\pinlabel $\cdots$ at 85 -10
\pinlabel $n\text{-}2$ at 118 -10
\pinlabel $n\text{-}1$ at 135 -10
\endlabellist
\includegraphics[width=\textwidth]{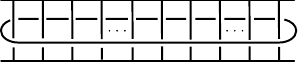}
\medskip
\caption{Isotopy from (a).}
\end{subfigure}

\bigskip

\begin{subfigure}[b]{0.32\textwidth}
\labellist
\small\pinlabel $(d\ d\text{+}1)$ at -20 15
\pinlabel $0$ at 5 -10
\pinlabel $1$ at 20 -10
\pinlabel $2$ at 35 -10
\pinlabel $3$ at 50 -10
\pinlabel $\cdots$ at 85 -10
\pinlabel $n\text{-}2$ at 118 -10
\pinlabel $n\text{-}1$ at 135 -10
\endlabellist
\includegraphics[width=\textwidth]{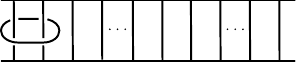}
\medskip
\caption{Moves N.}
\end{subfigure}
\hfill
\begin{subfigure}[b]{0.32\textwidth}
\labellist
\small\pinlabel $0$ at 5 -10
\pinlabel $1$ at 20 -10
\pinlabel $2$ at 35 -10
\pinlabel $3$ at 50 -10
\pinlabel $\cdots$ at 85 -10
\pinlabel $n\text{-}2$ at 118 -10
\pinlabel $n\text{-}1$ at 135 -10
\endlabellist
\includegraphics[width=\textwidth]{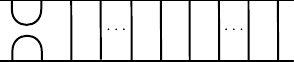}
\medskip
\caption{Apply Figure \ref{fig:useful} to (c).}
\end{subfigure}
\hfill
\begin{subfigure}[b]{0.32\textwidth}
\labellist
\small\pinlabel $0$ at 5 -10
\pinlabel $1$ at 20 -10
\pinlabel $2$ at 35 -10
\pinlabel $3$ at 50 -10
\pinlabel $\cdots$ at 85 -10
\pinlabel $n\text{-}2$ at 118 -10
\pinlabel $n\text{-}1$ at 135 -10
\endlabellist
\includegraphics[width=\textwidth]{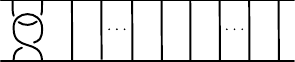}
\medskip
\caption{Isotopy from (d).}
\end{subfigure}
\end{figure}

\begin{figure}\ContinuedFloat
\centering
\begin{subfigure}[b]{0.32\textwidth}
\labellist
\small\pinlabel $(d\ d\text{+}1)$ at -20 20
\pinlabel $0$ at 5 -10
\pinlabel $1$ at 20 -10
\pinlabel $2$ at 35 -10
\pinlabel $3$ at 50 -10
\pinlabel $\cdots$ at 85 -10
\pinlabel $n\text{-}2$ at 118 -10
\pinlabel $n\text{-}1$ at 135 -10
\endlabellist
\includegraphics[width=\textwidth]{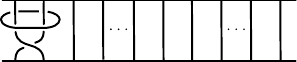}
\medskip
\caption{Apply Figure \ref{fig:useful} to (e).}
\end{subfigure}
\quad\qquad
\begin{subfigure}[b]{0.384\textwidth}
\labellist
\small\pinlabel $(d\ d\text{+}1)$ at 165 35
\pinlabel $0$ at 5 -10
\pinlabel $1$ at 20 -10
\pinlabel $2$ at 35 -10
\pinlabel $3$ at 50 -10
\pinlabel $\cdots$ at 85 -10
\pinlabel $n\text{-}2$ at 118 -10
\pinlabel $n\text{-}1$ at 135 -10
\endlabellist
\includegraphics[width=\textwidth]{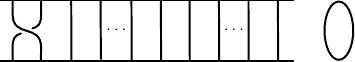}
\medskip
\caption{Perform (a)-(c) in reverse.}
\end{subfigure}

\caption{Obtaining II$_0$ using C and N. All strands are standardly colored.}
\label{fig:sheet ii0}
\end{figure}
\bigskip
\begin{figure}
\centering
\begin{subfigure}[b]{0.32\textwidth}
\labellist\small
\pinlabel $0$ at 3 -10
\pinlabel $1$ at 17 -10
\pinlabel $\cdots$ at 52 -10
\pinlabel $2i$ at 87 -10
\pinlabel $2i\text{+}1$ at 105 -10.5
\endlabellist
\includegraphics[width=\textwidth]{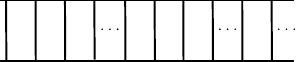}
\medskip
\caption{The trivial braid.}
\bigskip\bigskip
\end{subfigure}
\hfill
\begin{subfigure}[b]{0.32\textwidth}
\labellist\small
\pinlabel $0$ at 3 -10
\pinlabel $1$ at 17 -10
\pinlabel $\cdots$ at 52 -10
\pinlabel $2i$ at 87 -10
\pinlabel $2i\text{+}1$ at 105 -10.5
\endlabellist
\includegraphics[width=\textwidth]{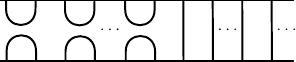}
\medskip
\caption{Use $i$ stabilizing unknots to open the strands as in Figure \ref{fig:sheet ii0} (a)-(d).}
\bigskip
\end{subfigure}
\hfill
\begin{subfigure}[b]{0.32\textwidth}
\labellist\small
\pinlabel $0$ at 3 -10
\pinlabel $1$ at 17 -10
\pinlabel $\cdots$ at 52 -10
\pinlabel $2i$ at 87 -10
\pinlabel $2i\text{+}1$ at 105 -10.5
\endlabellist
\includegraphics[width=\textwidth]{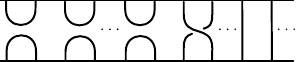}
\medskip
\caption{Use an $(i\text{+}1)$-th stabilizing unknot to apply moves as in Figure \ref{fig:sheet ii0}. Perform (a)-(b) in reverse.}
\end{subfigure}

\caption{Obtaining II$_i$ using C and N. All strands are standardly colored.}
\label{fig:sheet iii}
\end{figure}
\bigskip

\begin{figure}
\centering

\begin{subfigure}[b]{0.32\textwidth}
\labellist\small
\pinlabel $0$ at 5 -10
\pinlabel $1$ at 20 -10
\pinlabel $\cdots$ at 47 -10
\pinlabel $2d\text{-}5$ at 72 -10
\pinlabel $2d\text{-}4$ at 90 -10
\pinlabel $2d\text{-}3$ at 108 -10
\pinlabel $n\text{-}1$ at 135 -10
\endlabellist
\includegraphics[width=\textwidth]{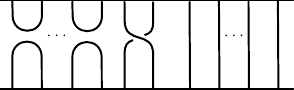}
\medskip
\caption{Setup as in Figure \ref{fig:sheet iii}.}
\end{subfigure}
\hfill
\begin{subfigure}[b]{0.32\textwidth}
\labellist\small
\pinlabel $0$ at 5 -10
\pinlabel $1$ at 20 -10
\pinlabel $\cdots$ at 47 -10
\pinlabel $2d\text{-}5$ at 72 -10
\pinlabel $2d\text{-}4$ at 90 -10
\pinlabel $2d\text{-}3$ at 108 -10
\pinlabel $n\text{-}1$ at 135 -10
\endlabellist
\includegraphics[width=\textwidth]{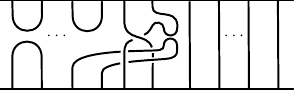}
\medskip
\caption{Isotopy from (a).}
\end{subfigure}
\hfill
\begin{subfigure}[b]{0.32\textwidth}
\labellist\small
\pinlabel $0$ at 5 -10
\pinlabel $1$ at 20 -10
\pinlabel $\cdots$ at 47 -10
\pinlabel $2d\text{-}5$ at 72 -10
\pinlabel $2d\text{-}4$ at 90 -10
\pinlabel $2d\text{-}3$ at 108 -10
\pinlabel $n\text{-}1$ at 135 -10
\endlabellist
\includegraphics[width=\textwidth]{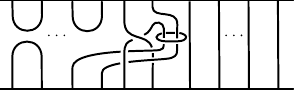}
\medskip
\caption{Apply Figure \ref{fig:useful}.}
\end{subfigure}

\bigskip

\begin{subfigure}[b]{0.32\textwidth}
\labellist\small
\pinlabel $0$ at 5 -10
\pinlabel $1$ at 20 -10
\pinlabel $\cdots$ at 47 -10
\pinlabel $2d\text{-}5$ at 72 -10
\pinlabel $2d\text{-}4$ at 90 -10
\pinlabel $2d\text{-}3$ at 108 -10
\pinlabel $n\text{-}1$ at 135 -10
\endlabellist
\includegraphics[width=\textwidth]{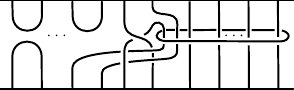}
\medskip
\caption{Moves N.}
\end{subfigure}
\hfill
\begin{subfigure}[b]{0.32\textwidth}
\labellist\small
\pinlabel $0$ at 5 -10
\pinlabel $1$ at 20 -10
\pinlabel $\cdots$ at 47 -10
\pinlabel $2d\text{-}5$ at 72 -10
\pinlabel $2d\text{-}4$ at 90 -10
\pinlabel $2d\text{-}3$ at 108 -10
\pinlabel $n\text{-}1$ at 135 -10
\endlabellist
\includegraphics[width=\textwidth]{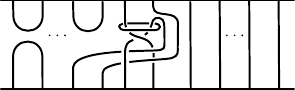}
\medskip
\caption{Isotopy from (d).}
\end{subfigure}
\hfill
\begin{subfigure}[b]{0.32\textwidth}
\labellist\small
\pinlabel $0$ at 5 -10
\pinlabel $1$ at 20 -10
\pinlabel $\cdots$ at 47 -10
\pinlabel $2d\text{-}5$ at 72 -10
\pinlabel $2d\text{-}4$ at 90 -10
\pinlabel $2d\text{-}3$ at 108 -10
\pinlabel $n\text{-}1$ at 135 -10
\endlabellist
\includegraphics[width=\textwidth]{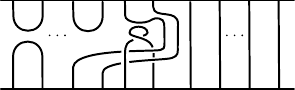}
\medskip
\caption{Apply Figure \ref{fig:useful}.}
\end{subfigure}
\end{figure}

\begin{figure}\ContinuedFloat
\centering
\begin{subfigure}[b]{0.32\textwidth}
\labellist\small
\pinlabel $0$ at 5 -10
\pinlabel $1$ at 20 -10
\pinlabel $\cdots$ at 47 -10
\pinlabel $2d\text{-}5$ at 72 -10
\pinlabel $2d\text{-}4$ at 90 -10
\pinlabel $2d\text{-}3$ at 108 -10
\pinlabel $n\text{-}1$ at 135 -10
\endlabellist
\includegraphics[width=\textwidth]{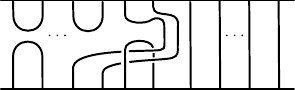}
\medskip
\caption{Isotopy from (f).}
\end{subfigure}
\hfill
\begin{subfigure}[b]{0.32\textwidth}
\labellist\small
\pinlabel $0$ at 5 -10
\pinlabel $1$ at 20 -10
\pinlabel $\cdots$ at 47 -10
\pinlabel $2d\text{-}5$ at 72 -10
\pinlabel $2d\text{-}4$ at 90 -10
\pinlabel $2d\text{-}3$ at 108 -10
\pinlabel $n\text{-}1$ at 135 -10
\endlabellist
\includegraphics[width=\textwidth]{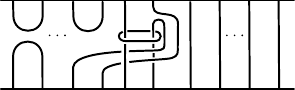}
\medskip
\caption{Apply Figure \ref{fig:useful}.}
\end{subfigure}
\hfill
\begin{subfigure}[b]{0.32\textwidth}
\labellist\small
\pinlabel $0$ at 5 -10
\pinlabel $1$ at 20 -10
\pinlabel $\cdots$ at 47 -10
\pinlabel $2d\text{-}5$ at 72 -10
\pinlabel $2d\text{-}4$ at 90 -10
\pinlabel $2d\text{-}3$ at 108 -10
\pinlabel $n\text{-}1$ at 135 -10
\endlabellist
\includegraphics[width=\textwidth]{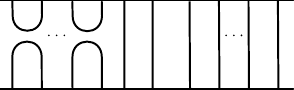}
\medskip
\caption{Perform (a)-(d) in reverse.}
\end{subfigure}

\caption{Obtaining II$_{d-3}$. All strands are standardly colored. }
\label{fig:sheet iid}
\end{figure}

\subsubsection{Move III}

III$_{i,j}$ can be carried out as in Figure \ref{fig:sheet 3}.

\begin{figure}
\centering

\begin{subfigure}[b]{0.23\textwidth}
\labellist\small
\pinlabel $i$ at 14 -9.5
\pinlabel $i\text{+}1$ at 31 -10
\pinlabel $\cdots$ at 65 -10
\pinlabel $j$ at 99 -10
\pinlabel $j\text{+}1$ at 116 -10
\endlabellist
\includegraphics[width=\textwidth]{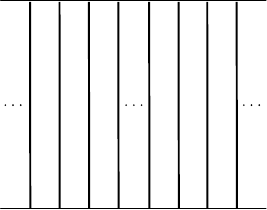}
\medskip
\caption{The trivial braid.}
\bigskip\bigskip\bigskip
\end{subfigure}
\hfill
\begin{subfigure}[b]{0.23\textwidth}
\labellist\small
\pinlabel $i$ at 14 -9.5
\pinlabel $i\text{+}1$ at 31 -10
\pinlabel $\cdots$ at 65 -10
\pinlabel $j$ at 99 -10
\pinlabel $j\text{+}1$ at 116 -10
\endlabellist
\includegraphics[width=\textwidth]{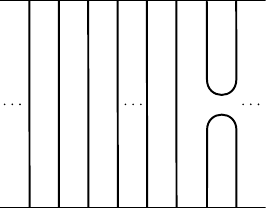}
\medskip
\caption{Use $j+1$ stabilizing unknots to open strands $j$ and $j\text{+}1$ as in Figure \ref{fig:sheet ii0}.}
\bigskip
\end{subfigure}
\hfill
\begin{subfigure}[b]{0.23\textwidth}
\labellist\small
\pinlabel $i$ at 14 -9.5
\pinlabel $i\text{+}1$ at 31 -10
\pinlabel $\cdots$ at 65 -10
\pinlabel $j$ at 99 -10
\pinlabel $j\text{+}1$ at 116 -10
\endlabellist
\includegraphics[width=\textwidth]{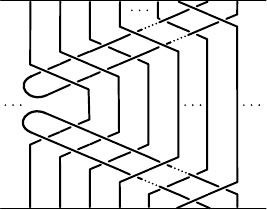}
\medskip
\caption{Isotopy from (c).}
\bigskip\bigskip\bigskip
\end{subfigure}
\hfill
\begin{subfigure}[b]{0.23\textwidth}
\labellist\small
\pinlabel $i$ at 14 -9.5
\pinlabel $i\text{+}1$ at 31 -10
\pinlabel $\cdots$ at 65 -10
\pinlabel $j$ at 99 -10
\pinlabel $j\text{+}1$ at 116 -10
\endlabellist
\includegraphics[width=\textwidth]{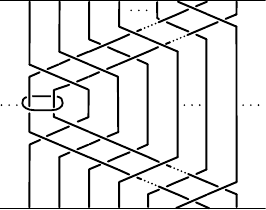}
\medskip
\caption{Apply Figure \ref{fig:useful}, and use isotopy and move N to remove the stabilizing unknot.}
\end{subfigure}

\caption{Obtaining III$_{i,j}$ using N and C. All strands are standardly colored.}
\label{fig:sheet 3}
\end{figure}

\subsubsection{Move IV}
The strategy of Fig 5(a) of \cite{pie95} can be used to produce move IV, as in Figure \ref{fig:sheet-iv}.

\begin{figure}
\centering

\begin{subfigure}[b]{0.3\textwidth}
\labellist\small
\pinlabel $0$ at 5 -10
\pinlabel $2d\text{-}10$ at 35 -10
\pinlabel $2d\text{-}4$ at 98 -10
\pinlabel $n\text{-}1$ at 136 -10
\tiny\pinlabel $\cdots$ at 16 -10
\pinlabel $\cdots$ at 119 -10
\endlabellist
\includegraphics[width=\textwidth]{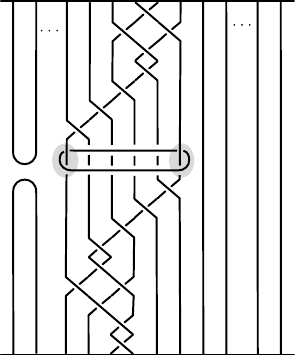}
\medskip
\caption{Starting from the trivial braid, use $d-3$ stabilizing unknots to open strands 0 through $2d\text{-}11$. Use an extra unknot colored $(3\ 2d\text{-}2)$, isotopy, and move N to attain the depicted state.}
\end{subfigure}
\hfill
\begin{subfigure}[b]{0.3\textwidth}
\labellist\small
\pinlabel $0$ at 5 -10
\pinlabel $2d\text{-}10$ at 35 -10
\pinlabel $2d\text{-}4$ at 98 -10
\pinlabel $n\text{-}1$ at 136 -10
\tiny\pinlabel $\cdots$ at 16 -10
\pinlabel $\cdots$ at 119 -10
\endlabellist
\includegraphics[width=\textwidth]{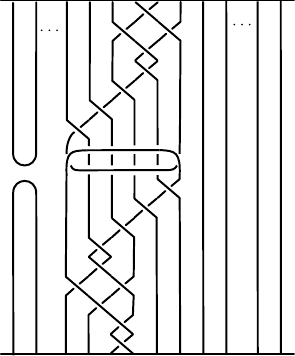}
\medskip
\caption{Move C in each shaded region of (a).}
\bigskip\bigskip\bigskip\medskip
\end{subfigure}
\hfill
\begin{subfigure}[b]{0.342\textwidth}
\labellist\small
\pinlabel $0$ at 5 -10
\pinlabel $2d\text{-}10$ at 35 -10
\pinlabel $2d\text{-}4$ at 98 -10
\pinlabel $n\text{-}1$ at 136 -10
\tiny\pinlabel $\cdots$ at 16 -10
\pinlabel $\cdots$ at 119 -10
\endlabellist
\includegraphics[width=\textwidth]{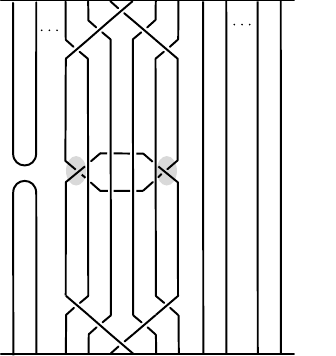}
\medskip
\caption{Isotopy from (b).}
\bigskip\bigskip\bigskip\bigskip\medskip
\end{subfigure}
\end{figure}

\begin{figure}\ContinuedFloat
\centering
\begin{subfigure}[b]{0.3\textwidth}
\labellist\small
\pinlabel $0$ at 5 -10
\pinlabel $2d\text{-}10$ at 35 -10
\pinlabel $2d\text{-}4$ at 98 -10
\pinlabel $n\text{-}1$ at 136 -10
\tiny\pinlabel $\cdots$ at 16 -10
\pinlabel $\cdots$ at 119 -10
\endlabellist
\includegraphics[width=\textwidth]{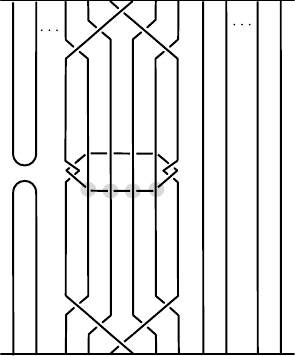}
\medskip
\caption{Move C in each shaded region of (c).}
\end{subfigure}
\hfill
\begin{subfigure}[b]{0.3\textwidth}
\labellist\small
\pinlabel $0$ at 5 -10
\pinlabel $2d\text{-}10$ at 35 -10
\pinlabel $2d\text{-}4$ at 98 -10
\pinlabel $n\text{-}1$ at 136 -10
\tiny\pinlabel $\cdots$ at 16 -10
\pinlabel $\cdots$ at 119 -10
\endlabellist
\includegraphics[width=\textwidth]{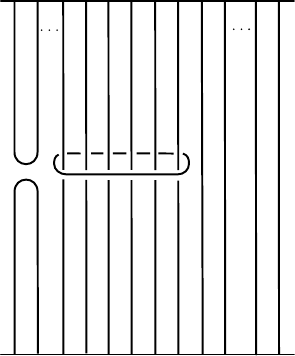}
\medskip
\caption{Move N in each shaded region of (d), and isotopy.}
\end{subfigure}
\hfill
\begin{subfigure}[b]{0.342\textwidth}
\labellist\small
\pinlabel $0$ at 5 -10
\pinlabel $2d\text{-}10$ at 35 -10
\pinlabel $2d\text{-}4$ at 98 -10
\pinlabel $n\text{-}1$ at 136 -10
\tiny\pinlabel $\cdots$ at 16 -10
\pinlabel $\cdots$ at 119 -10
\endlabellist
\includegraphics[width=\textwidth]{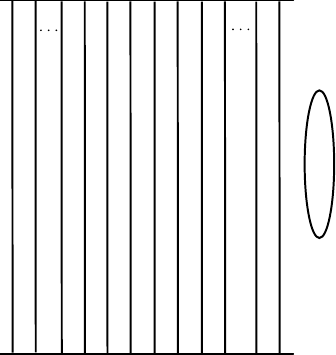}
\medskip
\caption{Move N and isotopy to remove the stabilizing unknot.}
\end{subfigure}

\caption{Obtaining IV using N and C. Strands are colored as in Figure \ref{fig:iv}.}
\label{fig:sheet-iv}
\end{figure}

\subsection{Proof of Corollary \ref{cor:lens}}\label{pf:lens}

Let $L$ be a degree $d$ colored link.

We first prove part (a). As the Heegaard genus of a simple branched cover of a link $L$ is given by (bridge number of $L$)$ - (d-1)$, degree $d$ simple covers of $d$-bridge links have Heegaard genus 1 and thus are lens spaces $L(p,q)$.

In order to determine $p$ and $q$, use C, N, and isotopy to write $L$ as as the stabilization to degree $d$ of a degree 3 colored link as in Lemma \ref{lem:stab}. The resulting link is a plat closure of a braid, which is a word in the generators $s_1, x_{2d-6}, x_{2d-4}, x_{2d-3}, \ldots, x_{n-2}$. Instances of $s_1$ can be successively removed, starting from the bottom, by performing the isotopy and moves C in Figure \ref{fig:lens-s}.

\begin{figure}
\centering
\begin{subfigure}[b]{0.3\textwidth}\centering
\labellist\footnotesize
\pinlabel $\cdots$ at 6 -10
\pinlabel $(23)$ at 27 -10
\pinlabel $(12)$ at 51 -10
\pinlabel $(12)$ at 74 -10
\endlabellist
\includegraphics[width=0.45\textwidth]{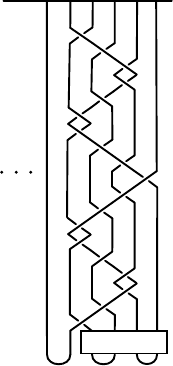}
\bigskip
\caption{The last instance of $s_1$.}
\end{subfigure}
\quad
\begin{subfigure}[b]{0.3\textwidth}\centering
\labellist\footnotesize
\pinlabel $\cdots$ at 6 -10
\pinlabel $(23)$ at 28 -10
\pinlabel $(12)$ at 68 -10
\endlabellist
\includegraphics[width=0.45\textwidth]{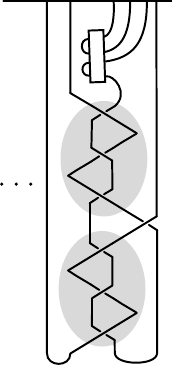}
\bigskip
\caption{Isotopy from (a).}
\end{subfigure}
\quad
\begin{subfigure}[b]{0.3\textwidth}\centering
\labellist\footnotesize
\pinlabel $\cdots$ at 6 -10
\pinlabel $(23)$ at 27 -10
\pinlabel $(12)$ at 72 -10
\endlabellist
\includegraphics[width=0.45\textwidth]{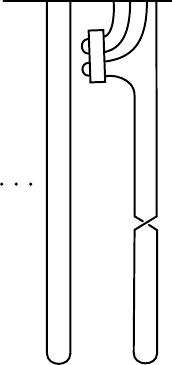}
\bigskip
\caption{Move C on each shaded region.}
\end{subfigure}

\caption{Removing $s_1$ in the $d$-bridge case using C and link isotopy.}
\label{fig:lens-s}
\end{figure}

This process can be repeated until we are left with a word in only $x_{2d-6}, x_{2d-4}, x_{2d-3}, x_{2d-2}$. This is a stabilization to degree $d$ of the double branched cover over a two-bridge link of some type $(p,q)$. Thus, the cover will be the lens space $L(p,q)$.

For part (b), we can proceed as in part (a) to eliminate each generator except $s_1$, which can be eliminated as in Figure \ref{fig:dbc-s}. We are left with the stabilization to degree $d$ of the double branched cover over a three-bridge link.

\begin{figure}
\centering
\begin{subfigure}[b]{0.15\textwidth}
\labellist\footnotesize
\pinlabel $(3\ 4)$ at 0 177
\endlabellist
\includegraphics[width=\textwidth]{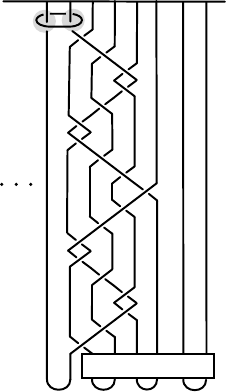}
\caption{The last instance of $s_1$, with colored unknot.}
\end{subfigure}
\qquad
\begin{subfigure}[b]{0.15\textwidth}
\labellist\footnotesize
\pinlabel $(2\ 4)$ at 4 160
\endlabellist
\includegraphics[width=\textwidth]{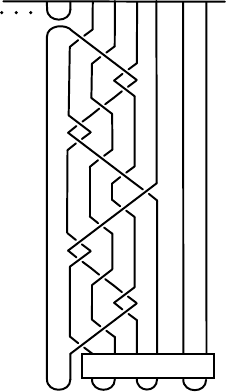}
\caption{C in shaded regions of (a).}
\bigskip\medskip\smallskip
\end{subfigure}
\qquad
\begin{subfigure}[b]{0.15\textwidth}
\labellist\footnotesize
\pinlabel $(2\ 4)$ at 120 50
\endlabellist
\includegraphics[width=\textwidth]{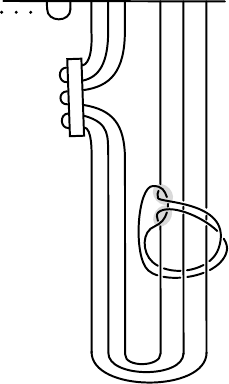}
\caption{Isotopy from (b).}
\bigskip\medskip\smallskip
\end{subfigure}
\qquad
\begin{subfigure}[b]{0.15\textwidth}
\labellist\footnotesize
\pinlabel $(2\ 4)$ at 120 56
\endlabellist
\includegraphics[width=\textwidth]{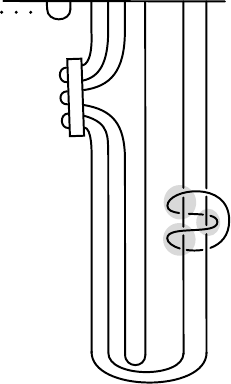}
\caption{C in shaded regions of (c).}
\bigskip\medskip\smallskip
\end{subfigure}
\qquad
\begin{subfigure}[b]{0.15\textwidth}
\labellist\footnotesize
\pinlabel $(2\ 4)$ at 0 90
\endlabellist
\includegraphics[width=\textwidth]{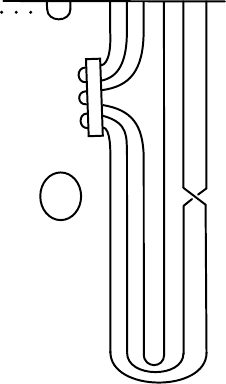}
\caption{C and isotopy from (d).}
\bigskip\medskip\smallskip
\end{subfigure}
\caption{Removing $s_1$ in the $d+1$-bridge case. The tops of the leftmost pair of strands are colored $(2\ 3)$. The tops of all other strands are colored $(1\ 2)$ unless otherwise noted. As no instances of $x_{2d-7}$ or $x_{2d-9}$ appear, the component colored $(3\ 4)$ is an unknot and can be moved to the position in (a) by isotopy and N.}
\label{fig:dbc-s}
\end{figure}

\bibliography{coveringmoves.bib}


\end{document}